%% file: main.tex
\documentclass{article}
\usepackage[utf8]{inputenc}
\usepackage{tikz-cd}
\usepackage{amsmath}
\usepackage{mathtools}
\usepackage{stmaryrd}
\usepackage{amsfonts}
\usepackage{amssymb}
\usepackage{amsthm}
\usepackage{mathrsfs}
\usepackage{hyperref} 
\usepackage{enumitem}
\newlist{enumroman}{enumerate}{1}
\setlist[enumroman]{font=\normalfont,label=(\roman*),leftmargin=0.3in}
\numberwithin{equation}{section}

\newtheorem{theorem}{Theorem}[section]

\newtheorem{lemma}[theorem]{Lemma}
\newtheorem{proposition}[theorem]{Proposition}
\newtheorem{corollary}[theorem]{Corollary}
\theoremstyle{definition}
\newtheorem{definition}[theorem]{Definition}
\newtheorem{example}[theorem]{Example}

\newtheorem{remark}[theorem]{Remark}

\newcommand{\norm}[1]{\left\lVert#1\right\rVert}
\usepackage{titling}
\newcommand{\subtitle}[1]{%
  \posttitle{%
    \par\end{center}
    \begin{center}\large#1\end{center}
    \vskip0.5em}%
}
\newcommand{\defn}[1]{\textbf{#1}}
\newcommand{\RN}[1]{%
  \textup{\uppercase\expandafter{\romannumeral#1}}%
}
\usetikzlibrary{matrix,arrows,backgrounds}
\tikzset{
curarrow/.style={
rounded corners=8pt,
execute at begin to={every node/.style={fill=red}},
to path={-- ([xshift=50pt]\tikztostart.center)
  |- (#1) 
  -| ([xshift=-40pt]\tikztotarget.center)
  -- (\tikztotarget)}
  }
}
\usepackage[title]{appendix}

\title{Bow varieties as symplectic reductions of $T^*(GL_n/P)$}
\author{Yibo Ji\\[1ex]\normalsize with an appendix by Matthias Franz}
\date{}
\begin{document}

\maketitle
\begin{abstract}
    Cherkis bow varieties were introduced as ADHM type description of moduli space of instantons on the Taub-NUT space equivariant under a cyclic group action. They are also models of Coulomb branches of quiver gauge theories of affine type $A$. In this paper, we realize each bow variety with torus fixed points as a symplectic reduction of a cotangent bundle of a partial flag variety by a unipotent group, and find a slice of this action. By this description, we calculate the equivariant cohomology (and ordinary cohomology) of some of them and answer some questions raised before. This also uses a new result about circle-equivariant cohomology proven
in an appendix. We also give an explicit generalized Mirkovic-Vybornov isomorphism for bow varieties in the appendix.
\end{abstract}

\section{Introduction}

Bow varieties were introduced as moduli spaces of $U(n)$-instantons on the Taub-NUT space equivariant under a cyclic group $\mathbb{Z}/l\mathbb{Z}$-action introduced by \cite{Cherkis:2011}. In a sequence of articles (\cite{NakajimaTakayama:2017} and \cite{Takayama:2016}), they are presented as algebraic varieties with hyper-K\"{a}hler structure via a quiver description, carrying a residual action of a torus $\mathbb{T}$. In  recent articles \cite{Nakajima:2021} and \cite{RimanyiShou:2020}, Rimanyi and Nakajima proved that there is a bijection between the $\mathbb{T}$-fixed points on certain bow varieties (that we will call ``pointful'') and ${0,1}$-matrices. We focus on the geometry and equivariant cohomology of this class in this paper. Here is one of the main results.

\begin{theorem}\label{mainthm} (For the proof, please refer to Theorem~\ref{pointful} and the preceding argument.)
    Each pointful bow variety is a symplectic reduction of the cotangent bundle of a partial flag variety by a unipotent group. Moreover, we have a slice of this action. Hence, every pointful bow variety can be viewed as a symplectic subvariety of a cotangent bundle of a flag variety.
\end{theorem}

In \S~\ref{section4}, we focus on pointful bow varieties whose fixed points correspond to $2$-row $\{0,1\}$-matrices. We calculate the equivariant cohomology (and ordinary cohomology) of some of them (\S~\ref{calforordcoh} and \S~\ref{calforequco}). Then we will see that, unlike for quiver varieties, the localization map to fixed points does not give us an injective map of equivariant cohomology for every bow variety, and the ordinary Kirwan surjectivity does not hold (\S~\ref{calforequco}). This answers a question in \cite{RimanyiShou:2020}. But, using Appendix~\ref{appendix2} (written by Matthias Franz), we can show that equivariant cohomology of $2$-row bow varieties is concentrated in even degrees (Theorem~\ref{secondcal}).

Bow varieties are Coulomb branches (\cite{NakajimaTakayama:2017}). There are other models called generalized affine Grassmannian slices introduced in \cite{BravermanFinkelbergNakajima:2018}. Bow varieties are generalized version of Nakajima type A quiver varieties and generalized affine Grassmannian slices are generalized version of affine Grassmannian slices. In \cite{MirkovićVybornov:2008}, Mirkovic-Vybornov isomorphism are introduced to give explicit isomorphisms between Nakajima type A quiver varieties, slices in the cotangent bundles of flag varieties, and slices in affine Grassmannian. In \S~\ref{MVyiso}, we will give a generalized version of Mirkovic-Vybornov isomorphisms bridgeing the slices appearing in symplectic reductions and generalized affine Grassmannian slices explicitly.

\section{Definition of Bow Varieties}\label{section1}
We follow the construction from \cite{RimanyiShou:2020}. The bow varieties are built from two building blocks. The first is the symplectic vector space $T^*Hom(V_1,V_2)\cong Hom(V_1,V_2)\times Hom(V_2,V_1)$, drawn as
\[\begin{tikzcd}
V_1\arrow[r,shift left,"E"]&V_2\arrow[l,shift left,"D"]
\end{tikzcd}\]
and called a \defn{two-way part}.

The second one is called a \defn{triangle part}, drawn as 
\[\begin{tikzcd}
V_2\arrow[out=120,in=60,loop,looseness=3,"B_2"]&&V_1\arrow[out=120,in=60,loop,looseness=3,"B_{1}"]\arrow[ll,"A"]\arrow[dl,"b"]\\&\mathbb{C}\arrow[lu,"a"]&
\end{tikzcd}\]
where we require $B_2A-AB_1+ab=0$ and the following two ``stability'' conditions:
\begin{itemize}
    \item If $S$ is a subspace of $V_1$ such that $B_1(S)\subset V_1$, $b(S)=0$, and $A(S)=0$, we have $S=0$.
    \item If $S$ is a subspace of $V_2$ such that $B_2(S)\subset S$ and $A(V_1)+a(\mathbb{C})\subset S$, we have $S=V_2$.
\end{itemize}
So the triangle part is a subset of $End(V_1)\times End(V_2)\times Hom(V_1,V_2)\times V_1^*\times V_2$. By \cite[Proposition~2.2]{NakajimaTakayama:2017}, the stability conditions arise from GIT, so we know that a triangle part is a locally closed subset of an affine space. Later in Lemma~\ref{slice_on_triangle}, we will show that each triangle part is actually affine.

Now introduce the \defn{handsaw quiver diagram}
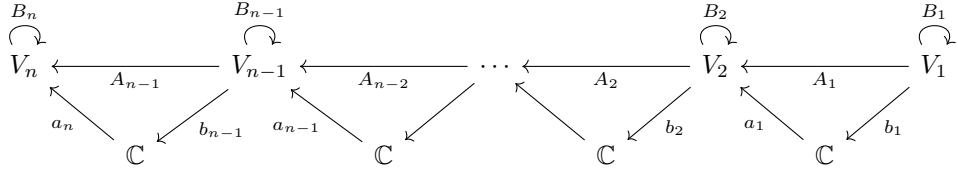
\[\begin{tikzcd}
V_n\arrow[out=120,in=60,loop,looseness=3,"B_n"]&&V_{n-1}\arrow[dl,"b_{n-1}"]\arrow[ll,"A_{n-1}"]\arrow[out=120,in=60,loop,looseness=3,"B_{n-1}"]&&\cdots\arrow[dl]\arrow[ll,"A_{n-2}"]&&V_2\arrow[out=120,in=60,loop,looseness=3,"B_2"]\arrow[ll,"A_2"]\arrow[dl,"b_2"]&&V_1\arrow[out=120,in=60,loop,looseness=3,"B_{1}"]\arrow[ll,"A_1"]\arrow[dl,"b_1"]\\&\mathbb{C}\arrow[ul,"a_n"]&&\mathbb{C}\arrow[ul,"a_{n-1}"] &&\mathbb{C}\arrow[lu]&&\mathbb{C}\arrow[lu,"a_1"]&
\end{tikzcd}\]
We are going to define the \defn{handsaw quiver variety} (\cite{Nakajima:2012}). The definition will be inductive, and the existence of two functions in the definition just makes the induction step easily described. A handsaw quiver variety and its attached two functions will be denoted by \[(HS(v_n,\cdots,v_1),B_{int,v_n,\cdots,v_1},B_{end,v_n,\cdots,v_1})\] where $B_{int,v_n,\cdots,v_1}$ (resp. $B_{end,v_n,\cdots,v_1})$) is $\mathfrak{gl}_{v_1}$-valued (resp. $\mathfrak{gl}_{v_n}$-valued) function on $HS(v_n,\cdots,v_1)$. It will be equipped with a natural $GL_{v_n}\times GL_{v_1}$ action. When $n=2$, it's just a triangle part equipped with standard $B_1,B_2$ and the group action just comes from $V_1$ and $V_2$.

When $n>2$, we construct it via a GIT quotient:
\[(HS(v_n,v_{n-1})\times_{B_{int,v_n,v_{n-1}}=B_{end,v_{n-1},\cdots,v_1}} HS(v_{n-1},\cdots, v_1)\sslash GL_{v_{n-1}}, B_{int,v_{n-1},\cdots,v_1},B_{end,v_n,v_{n-1}}) \]
The action of $GL_{v_n}$ (resp. $GL_{v_1}$) comes from $HS(v_n,v_{n-1})$ (resp. $HS(v_{n-1},\cdots,v_1)$).

\begin{remark}
    For simplification, we call $B_{end,v_n,v_{n-1}}$ (resp. $B_{int,v_{n-1},\cdots,v_1}$)the \defn{end function} (resp. the \defn{initial function}) when there is no ambiguity.
\end{remark}

\begin{remark}
 Because we know these function explicitly for the starting traingles and we build handsaw quiver varieties via GIT quotient step by step, the two functions in the definition of handsaw varieties are determined by the variety itself and the $GL_{v_n} \times GL_{v_1}$ action. Actually, we will later show that there is a symplectic structure on the handsaw quiver varieties and the $GL_{v_n} \times GL_{v_1}$ action is Hamiltonian. The two functions in the definition are just moment maps attached to the Hamiltonian action. We are using a tuple to define it to make the statement of the construction easier. From now on, we can just use the variety itself instead of the tuple to refer to the bow variety and call the first (resp. second) function the initial function (resp. end function). In \S~\ref{section2}, we will show that these two functions come from moment maps.
\end{remark}

\begin{remark}
    A handsaw variety is just a bow variety defined in \cite{NakajimaTakayama:2017} consisting only of triangles.
\end{remark}

\defn{Bow varieties} in this paper are defined to be
\[M(\vec{v},\vec{w})\times_{\pi_{\vec{w}}=B_{end,u_n,\cdots, u_1,u_0}} HS(u_n,\cdots,u_1,u_0)\sslash GL_{u_n}\]
where $M(\vec{v},\vec{w})$ is the standard notation for type $A$ Nakajima quiver varieties (with the stability parameter $(-1,-1,\cdots,-1)$), $\vec{w}=(0,\cdots,0,u_n)$ is the dimension vector of the framed vector spaces and $\mu_{\vec{w}}$ is the moment map for the $GL_{u_n}$ action on it. Bow varieties in this paper are those in \cite{RimanyiShou:2020} (which is the case that complex parameter is $\Vec{0}$ and the stability parameter is $(-1,\cdots,1)$ in \cite{NakajimaTakayama:2017}) represented by so-called $\defn{separated diagram}$ (see \cite[Remark~6.17]{RimanyiShou:2020}).

\begin{remark}
    When talking about general handsaw varieties, we use $(v_n,\cdots, v_1)$ to denote the dimension vector. When talking about bow varieties, we use $(u_n,\cdots ,u_1,u_0)$ to denote the dimension vector of the handsaw variety attached to it.
\end{remark}
If a bow variety satisfies all the following conditions, we will call it a \defn{good bow variety}:
\begin{itemize}
    \item $M(\vec{v},\vec{w})$ is a cotangent bundle $T^*Fl(l_1,\cdots ,l_m)$ of a flag variety;
    \item $u_n\geq u_{n-1}\geq \cdots \geq u_1 \geq u_0=0$.
\end{itemize}
If $u_i,r_j,1\leq i\leq m,1\leq j\leq n$ satisfy the following additional conditions, we will call it a \defn{pointful bow variety}  
\begin{itemize}
    \item $c_i:=u_{i+1}-u_i\leq m, \forall 1\leq i\leq n-1$;
    \item $r_{j+1}:=l_{j+1}-l_j\leq n, \forall 0\leq j\leq m-1$ where $l_0$ is set to be $0$;
    \item there is a $m\times(n-1)$ matrix with ${0,1}$ entries such that $r_i$ is the sum of the $i$-th row and $c_j$ is the sum of the $j$-th column.
\end{itemize}

If a handsaw quiver variety $HS(u_n,\cdots,u_1,u_0)$ satisfies 
 $u_n\geq u_{n-1}\geq \cdots \geq u_1 \geq u_0=0$, 
we will call it a \defn{good handsaw quiver variety}.

 Every pointful bow variety in \cite{RimanyiShou:2020} is isomorphic to a bow variety representing by a separated diagram via the Hanany-Witten isomorphisms (see \cite[\S~3.3]{RimanyiShou:2020} and \cite[Proposition~7.1]{NakajimaTakayama:2017}). So, using the definition, we won't miss any pointful bow variety up to isomorphism.

The definitions of good and pointful bow varieties come from \cite[Assumption~2.4]{RimanyiShou:2020}. Notice that there is a $\mathbb{C}^\times$ action on a handsaw variety for each $\mathbb{C}$ in the triangle. We add one more circle action with weight $1$ on $B_i$ and $b_i$ and weight $0$ on all other coordinates. This in all gives us a torus $\mathbb{T}=(\mathbb{C}^\times)^n$ acting on $HS(v_n,\cdots,v_1)$. By \cite[Proposition~4.9]{RimanyiShou:2020}, a bow variety defined in this paper has a torus fixed point if and only if it is a pointful bow variety. Good bow varieties are defined to be good because they can be realized as symplectic reductions of $T^*(GL_N/P)$ for some $N$ and a parabolic subgroup $P$ (as will be shown in \S~\ref{section2}).

\begin{remark}\label{compare}
    In \cite{RimanyiShou:2020}, a pointful bow variety (up to Hanany-Witten isomorphisms; see \cite[\S~2.6 and \S~3.3 ]{RimanyiShou:2020}) is determined by a row vector $\vec{r}=(r_1,\cdots,r_m)$ and a column vector $\vec{c}=(c_1,\cdots,c_n)$ such that there is a $\{0,1\}$ matrix of size $m\times n$ such that the sum of $i$-th row is $r_i$ and the sum of the $j$-th column is $c_j$. In our language, it is of the form
    \[T^*Fl(l_1,\cdots ,l_m)\times_{\pi=B_{end,u_{n},\cdots, u_1,u_0}} HS(u_{n},\cdots,u_1,u_0)\sslash GL_{u_n} \]
    where
    \begin{itemize}
        \item $l_i=\underset{1\leq j\leq i}{\sum}r_j$
        \item $u_{i}=l_m-\underset{i< j\leq n}{\sum}c_j $. In particular, $u_0 = l_m - \underset{0< j\leq n}{\sum}c_j = \underset{1\leq j\leq m}{\sum}r_j- \underset{0< j\leq n}{\sum}c_j = 0$. 
        Moreover, since $l_m = \sum_{1\leq i \leq m} r_i = \sum_{1\leq i \leq n} c_i$, we know that $u_{i}$ can be written as $\sum_{1\leq j\leq i} c_i $ as well.
        \item $\pi$ is the moment map with respect to the $GL_{l_n}$ action.
    \end{itemize}
\end{remark}

\begin{proposition}\label{triangleaffine}
The triangle part
\[
\left\{ (B_2, B_1, A, a, b) : 
\begin{array}{l}
    B_2 \in \operatorname{End}(V_2), \; B_1 \in \operatorname{End}(V_1), \; A \in \operatorname{Hom}(V_1, V_2), \\
    a \in V_2, \; b \in V_1^*, \; B_2 A - A B_1 + ab = 0, \\
    \text{such that:} \\
    \quad \bullet \; \text{If } S \text{ is a subspace of } V_1 \text{ with } B_1(S) \subset S, \; b(S) = 0, \; A(S) = 0, \text{ then } S = 0, \\
    \quad \bullet \; \text{If } S \text{ is a subspace of } V_2 \text{ with } B_2(S) \subset S \text{ and } A(V_1) + a(\mathbb{C}) \subset S, \text{ then } S = V_2 
\end{array}
\right\}
\]is affine.
\end{proposition}

\begin{proof}

This is the diagram we have 
\[\begin{tikzcd}
V_2\arrow[out=120,in=60,loop,looseness=3,"B_2"]&&V_1\arrow[out=120,in=60,loop,looseness=3,"B_{1}"]\arrow[ll,"A"]\arrow[dl,"b"]\\&\mathbb{C}\arrow[lu,"a"]&
\end{tikzcd}\]

By \cite[Lemma~2.18]{Takayama:2016}, we know that $A$ must be of full rank.

Let $v_i:=\dim V_i$. Notice that when take the dual of the above diagram, we will get the following diagram switching the dimension of $V_i$:
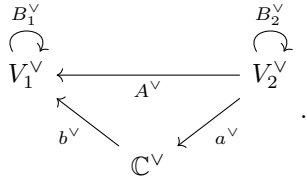
\[\begin{tikzcd}
V_1^\vee\arrow[out=120,in=60,loop,looseness=3,"B_1^\vee"]&&V_2^\vee\arrow[out=120,in=60,loop,looseness=3,"B_{2}^\vee"]\arrow[ll,"A^\vee"]\arrow[dl,"a^\vee"]\\&\mathbb{C}^\vee\arrow[lu,"b^\vee"]&
\end{tikzcd}.\]

To show the triangle part is affine, it reduces to consider the case that $v:=v_2-v_1\geq 0$ by the above dual operation. 

Since $v_2 - v_1 \geq 0$ and $A : V_1 \to V_2$ is of full rank, we know that $ker(A) = 0$. This tells that the first stability condition automatically holds. So we only need to deal with the second one. 

Denote $A(V_1) \cup \left( \underset{0\leq i\leq k}{\bigcup}B_2^{i} a(1)\right)$ by $W_k$. By the construction, we know that $W_{k},k\in \mathbb{N}$ is tower of vector spaces and $\dim W_k + 1\geq \dim W_{k+1}$. Since $B_2A - AB_1 + ab= 0$, we have another characterization
\[W_{k+1} = B_2W_{k} + AV_1.\] Hence, we know that once $W_k = W_{k+1}$, this tower stabilizes at this space. Choose $k_0$ to be the smallest number such that $W_{k_0} = W_{k_0+1}$. Then we know that $\dim W_{k_0} = \dim V_1 + k_0 + 1$ and $W_{k_0}$ is spanned by $AV_1$ and $B_2^{i}a(1), 0\leq i\leq k_0$.  The second stability condition is equivalent to say that $W_{k_0} = V_2$. Hence we know that $k_0 = v-1$ and $\begin{pmatrix} a(1)&B_2a(1)&\cdots &B_2^{v-1}a(1)&A \end{pmatrix}$ is of full rank. Since this matrix is a square matrix, we know that the stability condition could be transformed to the following
\[\mathcal{T}:=\begin{pmatrix} a(1)&B_2a(1)&\cdots &B_2^{v-1}a(1)&A \end{pmatrix}\in GL(V_2).\]

\begin{remark}
    When $v=0$, the above matrix is just $A$.
\end{remark}
Using $T$'s columns as a basis of $V_2$, we have the following lemma:

\begin{lemma}\cite[Prop~2.9]{Takayama:2016}\label{slice_on_triangle}
    The triangle part is isomorphic to 
    \begin{itemize}
        \item $GL_{v_2}\times \mathfrak{gl}_{v_1} \times T^*(\mathbb{C}^{v_1})$ when $v=0$.
        \item $GL_{v_2}\times \left\{M=\begin{pmatrix}0_{1 \times (v - 1)}&\gamma_1&-b\\ I_{v-1}&\gamma_2&0_{v-1\times v_2}\\ 0_{v_1 \times (v-1)}&\gamma_3&B_1\end{pmatrix}:\parbox{0.38\columnwidth}{$\gamma=\begin{pmatrix}\gamma_1\\\gamma_2\\\gamma_3\end{pmatrix}\in \mathbb{C}^{v_2},$\\ $b\in (\mathbb{C}^{v_1})^*, B_1\in End_\mathbb{C}(\mathbb{C}^{v_1})$}\right\}$ when $v>0$.
    \end{itemize} 
  
\end{lemma}
\begin{proof}
   We identify he matrix $\mathcal{T}$ as the element in $GL_{v_2}$ in both cases. 

   Consider the case that $v>0$. The bijection between the triangle part and the affine variety listed above is as follow:
   \begin{itemize}
       \item The first column of $\mathcal{T}$ is $a$.
       \item The last $v_2$ columns of $\mathcal{T}$ is $A$.
       \item $B_1$ is $B_1$ and $b$ is $b$.
       \item $B_2$ is $\mathcal{T}M\mathcal{T}^{-1}$.
   \end{itemize}

   Let us consider the case that $v=0$. The bijection between the triangle part and the affine variety listed above is as follow:
   \begin{itemize}
       \item $\mathfrak{gl}_{v_1}$ corresponds to $B_1$.
       \item $A$ is $\mathcal{T}$ .
       \item $a$ corresponds to $\mathbb{C}^{v_1}$.
       \item $b\mathcal{T}$ corresponds to $(\mathbb{C}^{v_1})^*$.
       \item $B_2=\mathcal{T}(B_1+ab\mathcal{T})\mathcal{T}^{-1}$
   \end{itemize}

   Here some explanation for why $M$ is of the shape described above. 
   
   The first $v - 1$ columns of $M$ are of the shape described because the first $v$ elements of the basis chosen are $a(1), B_2 a(1),\cdots, B_2^{v-1}a(1) $ and we are describing the action of $B_2$.

   The last $v_1$ columns of $M$ are of the shape described because $B_2 A = A B_1 - ab$.

\end{proof}
(Finishing Prop~\ref{triangleaffine}) Now, from the above lemma, we know that the triangle part is affine.
\end{proof}
\begin{remark}\label{action_on_triangle}
    In both cases, we are representing the triangle part as 
    \[GL_{v_2}\times S\]
    where $S$ is affine. The natural action of $GL_{v_2}\times GL_{v_1}$ is componentwise. On $GL_{v_2}$, it's exactly \[(g_2,g_1)(h)=(I_{v}\oplus g_1) h g_2^{-1}.\]  On $S$, the $GL_{v_2}$ action is trivial. Notice that the element $M$ in $S$ is characterized by
    \[ B_2 T = TM. \]
    The $GL_{v_2}$ action on $B_2$ (resp. T) is conjugation (left multiplication. Thus $M$ is invariant under $GL_{v_2}$ action. 
    
    The function $B_1$ is in $Mat_{v_1\times v_1}(\mathbb{C}[S])$ and $B_1(g_1(x))=g_1B_1(x)g_1^{-1}$.
    There exists $\tilde{B}_2=M\in Mat_{v_2\times v_2}(\mathbb{C}[S])$ such that $B_2(g_2,x)=g_2\tilde{B}_2g_2^{-1}$.
    
    \end{remark}

We can generalize the above remark to give a nice description of good handsaw varieties.

\begin{theorem}\label{construction1}
\begin{itemize}
    \item Every good handsaw variety $HS(u_n,\cdots, u_1,u_0)$ is isomorphic to
\[GL_{u_n}\times S\]
where $S$ is affine.
\item The natural $GL_{u_n}$ action on $GL_{u_n}\times S$ is componentwise, where the action on $GL_{u_n}$ is right multiplication by the inverse and the action on $S$ is trivial.
\item There exists a matrix function $\tilde{f}\in Mat_{u_n\times u_n}(\mathbb{C}[S])$ such that the end function $f$ of this handsaw variety satisfies $f(g,x)=g\tilde{f}(x)g^{-1}$.
\item In particular, if $u_n>u_{n-1}>\cdots>u_2 >u_1>u_0 = 0$, we can realize $S$ as a subvariety of $\mathfrak{gl}_{u_n}$.
\end{itemize}
\end{theorem}
\begin{proof}
    We proceed by induction. When $n=2$, this is exactly 
 Lemma~\ref{slice_on_triangle} and Remark~\ref{action_on_triangle}.

    Suppose we have shown that $HS(u_{n-1},\cdots, u_1)\cong GL_{u_{n-1}}\times S$.  The natural $GL_{u_{n-1}}$ action on $GL_{u_{n-1}}\times S$ is pointwise, where the action on $GL_{u_{n-1}}$ is right multiplication by the inverse and the action on $S$ is trivial. 

    We have the following decomposition: \begin{align*}&HS(u_n,u_{n-1})\times_{B_{int,u_n,u_{n-1}}=f}(GL_{u_{n-1}}\times S)\\=\ & 
    GL_{u_{n-1}}\times (HS(u_n,u_{n-1})\times_{B_{int,u_n,u_{n-1}}=\tilde{f}}(S))\end{align*}
    where the $GL_{u_{n-1}}$ action only acts on the first component.
    The map sends $(y,(g,x))$ to $(g,(g(y),x))$. To show it's an isomorphism, it suffices to show that it's well defined, i.e. if $B_{int,u_n,u_{n-1}}(y)=f(g,x)$, we have $B_{int,u_n,u_{n-1}}(g(y))=\tilde{f}(x)$. By definition, we know that $\tilde{f}(x)=gf(g,x)g^{-1}$. By Remark~\ref{action_on_triangle}, we know that $B_{int,u_n,u_{n-1}}(g(y))=gB_{int,u_n,u_{n-1}}(y)g^{-1}$. The combination of these two identities is exactly what we need. Hence the map given above is an isomorphism.
    
    Notice that $h(y,(g,x))$ is $(h(y),(gh^{-1},hx))=(hy,(gh^{-1},x))$ (the $GL_{u_{n-1}}$ action on $S$ is trivial). It's mapped to $(gh^{-1},(gy,x))$. So the action on the second component is trivial. After we quotient by this action, we know that $HS(u_n,u_{n-1},\cdots,u_1)$ is isomorphic to 
    \[HS(u_n,u_{n-1})\times_{B_{int,u_n,u_{n-1}}=\tilde{f}}(S) \]

    By Lemma~\ref{slice_on_triangle}, we know that $HS(u_n,u_{n-1})$ is isomorphic to $GL_{u_n}\times Y$ for some affine variety $Y$, and $B_{int,u_n,u_{n-1}}$ comes from $\mathbb{C}[Y]$. By associativity of fiber product, we know that $HS(u_n,u_{n-1},\cdots,u_1,u_0)$ is isomorphic to 
    \[GL_{u_n}\times (Y\times_{B_{int,u_n,u_{n-1}}=\tilde{f}}(S)) \]
    The latter part is an affine variety.
    
    The end function on this handsaw variety is $B_{end,u_n,u_{n-1}}$. By Remark~\ref{action_on_triangle}, there exists $\tilde{B}_{end,u_n,u_{n-1}}\in Mat_{u_n\times u_n}(\mathbb{C}[Y])$ such that \[B_{end,u_n,u_{n-1}}(g,y)=g\tilde{B}_{end,u_n,u_{n-1}}(y) g^{-1}\]. 
    
    By the statement of action on $HS(u_n,u_{n-1})$ in Remark~\ref{action_on_triangle}, we know the action of $GL_{u_n}$ on $GL_{u_n}$ is right multiplication by the inverse, and the action on the affine piece is trivial. 

    Notice that when $u_n>u_{n-1}>\cdots >u_1$, we know $S$ is a subvariety of $\mathfrak{gl}_{u_{n-1}}$ and  $Y$ is 
    \[\left\{M=\begin{pmatrix}0_{1\times (v-1)}&\gamma_1&b\\ I_{v-1}&\gamma_2&0_{v-1\times v_2}\\ 0_{v_1\times {v-1}}&\gamma_3&B_1\end{pmatrix}\in \mathfrak{gl}_{u_n}:\parbox{0.43\columnwidth}{$\gamma=\begin{pmatrix}\gamma_1\\\gamma_2\\\gamma_3\end{pmatrix}\in \mathbb{C}^{u_n}$,\\ $b\in (\mathbb{C}^{u_{n-1}})^*, B_1\in End_\mathbb{C}(\mathbb{C}^{u_{n-1}})$}\right\}.\]
    The fiber product $(Y\times_{B_{int,u_n,u_{n-1}}=\tilde{f}}(S))$ just replaces $B_1$ with $S$. So it is again a subvariety of $\mathfrak{gl}_{u_n}$.
\end{proof}
\begin{remark}\label{notation_slice}
    Let us denote $S$ from the above theorem by $S_{u_n,\cdots,u_1,u_0}\subset \mathfrak{gl}_{u_n}$. In \S~\ref{section4}, we will give an explicit description of it.
\end{remark}

\begin{corollary}
    Every good bow variety can be realized as a fibre product of $T^*Fl$ and some affine variety. 
    
    In particular, suppose we are given a bow variety of the following form:
    \[T^*Fl(l_1,l_2,\cdots,l_m)\times_{\pi=B_{end}}HS(u_n,\cdots,u_1,u_0) \] where 
    \begin{itemize}
        \item $l_1\leq l_2\leq \cdots \leq l_m=u_n$;
        \item When talk about the flag variety in this paper, we use the surjective flags to present each point inside it, i.e. each point corresponds to $\mathbb{C}^{l_m}\twoheadrightarrow \cdots \twoheadrightarrow \mathbb{C}^{l_1} $;
        \item $\pi$ is the moment map of the $GL_{l_m}$ action on $T^*Fl(l_1,l_2,\cdots,l_m)$;
        \item $u_n>\cdots>u_1 > u_0 = 0$.
    \end{itemize}
    Then the bow variety is isomorphic to 
    \[\pi^{-1}(S_{u_n,\cdots,u_1,u_0}) \]
    where $\mu$ is the moment map with respect to the $GL_{u_n}$ action on $T^*(Fl(l_1,\cdots,l_m))$.
\end{corollary}
\begin{proof}
    This will come from the construction of the bow variety.

Since there is no ambiguity, we just use $S$ to represent $S_{u_n,\cdots,u_1,u_0}$.

    In general, the bow variety is just 
    \[M(\vec{v},\vec{w})\times_{\pi_{\vec{w}}=B_{end}} (GL_{u_m}\times  S)\sslash GL_{u_m}=M(\vec{v},\vec{w})\times_{\pi_{\vec{w}}=B_{end}} S\]
    When we are in the special case from Theorem~\ref{construction1}, we know that $HS(u_n,\cdots, u_1)$ is $GL_{l_m}\times S$ where $S$ is a subvariety of $\mathfrak{gl}_{u_n}=\mathfrak{gl}_{l_m}$. So, by the construction, we can see that the bow variety is exactly $\mu^{-1}(S) $.
\end{proof}
In fact, by the following, for a good bow variety, we can always find another good bow variety isomorphic to it whose corresponding handsaw variety $HS(u_n,u_{n-1},\cdots,u_1,u_0)$ satisfies $u_n>u_{n-1}>\cdots >u_1 > u_0 = 0$. So we know that every good bow variety can be realized as a closed subvariety of a cotangent bundle of a flag variety.

\begin{lemma}\label{add_1}
    A good bow variety \[T^*Fl(l_1,l_2,\cdots,l_m)\times_{\pi=B_{end}}HS(u_n,\cdots,u_1,u_0),\]
with $1\leq i\leq m+1$, is isomorphic to 
 \[T^*Fl(l_1,l_2,\cdots,l_{m},l_m+n)\times_{\pi=B_{end}}HS(u_n+n,\cdots,u_i+i,\cdots,u_1+1, u_0).\]
\end{lemma}
\begin{proof}
     Consider the construction of $S_{u_n+n,\cdots,u_i+i,\cdots,u_1 + 1, u_0}$. 

     Consider the following two set
     \[I:=\{i:1\leq i\leq u_n+n,i\neq u_n+n+2-u_j-j,1\leq j\leq n\}\]
     \[J:=\{j:1\leq j\leq u_n+n,j\neq u_n+n-u_i-i,0\leq i\leq n-1\}\]
     Notice that $\#I=\#J=u_n$ and the $(I,J)$ submatrix of any matrix in $S_{u_n+n,\cdots,u_i+i,\cdots,u_1 + 1, u_0}$ is $I_{u_n}=I_{l_m}$.

     On the other hand, for any element $a\in T^*Fl(l_1,l_2,\cdots,l_{m},l_m+n)$, we have $rank(\pi(a))\leq l_m$. So we know that for any \[(a,x)\in T^*Fl(l_1,l_2,\cdots,l_{m},l_m+n)\times_{\pi=B_{end}}HS(u_n+n,\cdots,u_i+i,\cdots,u_1 + 1, u_0), \]  
     we have $rank(\pi(a))=l_m$. Moreover, we have a framed basis of the $l_m$ dimensional quotient of $a$ given by the columns indexed by $J$. Using this basis, we can find the endomorphism of this $l_m$ dimensional quotient is exactly of the form $B_{end}$ attached to $S_{u_n,\cdots,u_1, u_0}$.

     This gives us a map from \[T^*Fl(l_1,l_2,\cdots,l_{m},l_m+n)\times_{\pi=B_{end}}HS(u_n+n,\cdots,u_i+i,\cdots,u_1,u_0)\] to 
     \[ T^*Fl(l_1,l_2,\cdots,l_m)\times_{\pi=B_{end}}HS(u_n,\cdots,u_1,u_0)\]

     It's straightforward to verify that this map is an isomorphism.
\end{proof}
\begin{remark}
When the variety is a pointful bow variety, there is another proof for this lemma by Hanany-Witten isomorphisms and delicate analysis of separating lines (\cite[\S~2.4 and \S~2.6]{RimanyiShou:2020}). Interested readers may find the proof by themselves as an exercise.
\end{remark}
\section{Symplectic Structure on Bow Varieties}\label{section2}
In this section, we give an embedding of a pointful bow variety into the cotangent bundle of a flag variety, using a slice of a symplectic reduction. So the embedding makes the bow variety a symplectic subvariety of $T^*Fl$.

Before the main context, we introduce an argument which will be used frequently in this section:

\begin{lemma}\label{MW}
    Suppose $X$ is a symplectic manifold with a Hamiltonian action of $G$, a moment map $\pi$, and a regular value $f\in \mathfrak{g}^*$. Consider the symplectic reduction at $f$. Suppose $Y$ is a submanifold of $X$ such that the composition of the following maps is an isomorphism:
    \[Y\hookrightarrow \pi^{-1}(f)\to \pi^{-1}(f)\slash Stab_G(f).\]
    Then $Y$ can be viewed as the symplectic reduction via the above map and the symplectic form is just the pullback of that on $X$.
\end{lemma}
\begin{remark}
    These $Y$ are very rare when the group $G$ is not unipotent.
\end{remark}
\begin{proof}
    This basically involves the argument in the proof of symplectic reduction (see \cite{MarsdenWeinstein:1974}).

    Consider the following diagram:
    \[\begin{tikzcd}Y\arrow[r]&\pi^{-1}(f)\arrow[r]\arrow[d]&X\\& \tilde{X}:=\pi^{-1}(f)\slash Stab_G(f) &\end{tikzcd}.\]

    Suppose that the symplectic form on $X$ is $\omega$ and the symplectic form on $\tilde{X}$ is $\tilde{\omega}$. For a point $\tilde{p}\in \tilde{X}$, we may find its lifting $p$ in $Y$ by the isomorphism. Also, the isomorphism tells us that for any $p\in Y$, we have
    \[ T_p \mu^{-1}(f)=T_p Y\oplus T_p(\mu^{-1}(f))^\perp\] where $T_p(\mu^{-1}(f))^\perp$ consists of the vectors inside $T_p(\mu^{-1}(f))$ annihilating $T_p(\mu^{-1}(f))$ via the symplectic for $\tilde{\omega}$.
    Thus, for any two vectors $\tilde{\alpha},\tilde{\beta}\in T_{\tilde{p}
    }{\tilde{X}}$, we are able to pick their lifting in $T_p Y$: $\alpha,\beta$. Then by the construction of the quotient, we know that 
    \[\omega(\alpha,\beta)=\tilde{\omega}(\tilde{\alpha},\tilde{\beta}).\]

    So, the pullback of the symplectic form on $\tilde{X}$ via the isomorphism is exactly the restriction of that on $X$.
\end{proof}

\begin{remark}
    The lemma above is stated in the category of symplectic manifolds. However, in this paper, whenever we apply this lemma, both the action and the moment map are algebraic, and the group \( \operatorname{Stab}_G(f) \) always acts freely on the set \( \pi^{-1}(f) \). Therefore, we are justified in applying it within the algebraic context in this paper.
\end{remark}

We will call such $Y$ a \defn{slice} of the symplectic reduction of $X$ by $G$ at the level $f$.

Let us begin the main context by investigating a symplectic structure on the triangle parts. 

Notice that in the statement of Lemma~\ref{slice_on_triangle}, the first affine variety is $T^*GL_{v_2}\times T^*\mathbb{C}^{v_2}$. The symplectic structure on this triangle part is just the natural one on cotangent bundles.

Now focus on the second affine variety appearing in Lemma~\ref{slice_on_triangle}:

\[GL_{v_2}\times \left\{\begin{pmatrix}0_{1\times (v-1)}&\gamma_1&b\\ I_{v-1}&\gamma_2&0_{v-1\times v_2}\\ 0_{v_1\times (v-1)}&\gamma_3&B_1\end{pmatrix}:\gamma=\begin{pmatrix}\gamma_1\\\gamma_2\\\gamma_3\end{pmatrix}\in \mathbb{C}^{v_2}, b\in (\mathbb{C}^{v_1})^*, B_1\in End_\mathbb{C}(\mathbb{C}^{v_1})\right\}\]
Using right invariant vector fields to trivialize $T_*GL_{v_2}$ and the bilinear form $\left<X,Y\right>:=tr(XY)$ to identify $\mathfrak{g}$ with $\mathfrak{g}^*$, we may view $T^*GL_{v_2}$ as $G\times \mathfrak{g}$. So the above variety is a subvariety of $T^*GL_{v_2}$. Recall that in Remark~\ref{notation_slice}, we denote the second factor by $S_{v_2,v_1}$

Denote by $\mathbb{U}_i$ be the strict upper triangular matrix group in $GL_i$ and by $\mathbb{N}_i\subset \mathfrak{gl}_i$ the Lie algebra of $\mathbb{U}_i$

Consider the following unipotent subgroup of $Gl_{v_2}$:
\[U_{v_2,v_1}:=\left\{\begin{pmatrix}M&N\\0_{v_1 \times v}&I_{v_1}\end{pmatrix}:M\in \mathbb{U}_{v_2-v_1},\ N=\begin{pmatrix}N_1\\0_{1\times v_1}\end{pmatrix},\ N_1\in Mat_{(v_2-v_1-1)\times v_1}\right\}.\]

Let $f_{v_2,v_1}:=\begin{pmatrix}0_{1\times (v-1)}&0_{1 \times 1}&0_{1\times v_1}\\ I_{v-1}&0_{(v-1)\times 1}&0_{v-1\times v_2}\\ 0_{v_1\times (v-1)}&0_{(v_1)\times 1}&0_{v_1\times v_1}\end{pmatrix}$.

\begin{theorem}\label{base}\cite[Remark~3.4]{NakajimaTakayama:2017}
    The triangle part is a symplectic reduction of $T^*GL_{v_2}$ at $f_{v_2,v_1}$ by $U_{v_2,v_1}$, and the affine variety $GL_{v_2}\times S_{v_2,v_1}$ in Lemma~\ref{slice_on_triangle} is in fact a slice of this quotient.
\end{theorem}
\begin{proof}
    Notice that we can view $\mathfrak{u}^*$ as the transpose of $\mathfrak{u}$ using the form $\left<X,Y\right>$. Then we may view $f$ as an element inside $\mathfrak{u}^*$. It can be checked directly that the whole group $U$ is the stabilizer of $f_{v_2,v_1}$ (for any $u\in \mathfrak{u}$, $[u,f_{v_2,v_1}]$ (nonzero as a matrix) will vanish on the transpose part of $\mathfrak{u}$). 

    Recall that $T^* GL_{v_2}$ is of the form $GL_{v_2} \times Mat_{v_2 \times v_2}$. We view the second part as the dual of itself. Thus there is a restriction map $res : Mat_{v_2 \times v_2} \to \mathfrak{u}^*$ which just forgets all the zero entries in the transpose of $\mathfrak{u}$. The moment map of group action is just sending $(g,x)$ to $res(x)$.
    So the level set at $f_{v_2,v_1}$ is just $GL_{v_2}$ times the set of matrices of the following form:
    \[\begin{pmatrix}a &\gamma_1&b\\ I_{v-1}+T&\gamma_2& Q\\ 0_{v_1\times (v-1)}&\gamma_3&B_1\end{pmatrix}\]where \[\gamma=\begin{pmatrix}\gamma_1\\\gamma_2\\\gamma_3\end{pmatrix}\in Mat_{v_2\times 1},\ a\in Mat_{1\times {v-1}},\ T\in \mathbb{N}_{v-1}\]\[ b\in Mat_{1\times v_1},\ Q\in Mat_{(v-1)\times v_1},\ B_1\in  Mat_{v_1\times v_1}\]

    Denote the above level set by $GL_{v_2}\times P_{v_2,v_1}$. We are going to prove that this space is isomorphic to 
    \[GL_{v_2}\times S_{v_2,v_1}\times U_{v_2,v_1}\]
    where the action by $U_{v_2,v_1}$ is trivial on the first two coordinates and is left multiplication on the third coordinate. Then we'll know that $GL_{n}\times S_{v_2,v_1}$ is a slice of this quotient.

    In fact, it suffices to show that for any element $x$ in $P_{v_2,v_1}$ there exists a unique element $u\in U_{v_2,v_1}$ such that $uxu^{-1}\in S_{v_2,v_1}$. Let us pick \[\begin{pmatrix}a &\gamma_1&b\\ I_{v-1}+T&\gamma_2& Q\\ 0_{v_1\times (v-1)}&\gamma_3&B_1\end{pmatrix}\in P_{v_2,v_1},\begin{pmatrix}M&N\\0&I_{v_1}\end{pmatrix}\in U_{v_2,v_1} .\]
    The equation we want to solve is 
    \[\begin{pmatrix}M&N\\\begin{matrix}0&0\end{matrix}&I_{v_1}\end{pmatrix}\begin{pmatrix}a &\gamma_1&b\\ I_{v-1}+T&\gamma_2& Q\\ 0_{v_1\times (v-1)}&\gamma_3&B_1\end{pmatrix}=\begin{pmatrix}0_{1\times (v-1)}&\gamma'_1&b'\\ I_{v-1}&\gamma'_2&0_{v-1\times v_1}\\ 0_{v_1\times (v-1)}&\gamma'_3&B'_1\end{pmatrix}\begin{pmatrix}M&N\\0_{v_1\times v}&I_{v_1}\end{pmatrix}\]
   where $M\in \mathbb{U}_{v_2-v_1},N=\begin{pmatrix}N_1\\0_{1\times v_1}\end{pmatrix},N_1\in Mat_{(v_2-v_1-1)\times v_1}$. 
    The goal is to show that once we fix $a,\gamma_1,\gamma_2,\gamma_3,b,Q,b_1$ and also $T$, there exists a unique pair $(M,N)$ such that the above equation holds.

    We get that $\gamma_3'=\gamma_3$ and $B_1=B'_1$ from the last $v_1$ rows. Moreover, the first $v$ rows won't involve $B'_1$ and $\gamma'_3$ any more. So it suffices to show there exists a unique pair $(M,N)$ such that the first $v$ rows of the above equation hold.

    Let us focus on the upper left $v\times v$ part. We get the following equation:
    \[M\begin{pmatrix}a&\gamma_1\\I_{v-1}+T&\gamma_2\end{pmatrix}+\begin{pmatrix}0_{v\times(v-1)}&N\gamma_3\end{pmatrix}=\begin{pmatrix}0_{1\times(v-1)}&\gamma_1'\\I_{v-1}&\gamma_2'\end{pmatrix}M\]

Suppose $\vec{e}_1$ is the first standard basis element of $\mathbb{C}^v$. Then we have $M\vec{e}_1=\vec{e}_1$ since $M\in \mathbb{U}_{v_2-v_1}$. So the above equation is equivalent to
\[ M\begin{pmatrix}1&a&\gamma_1\\0_{(v-1)\times 1}&I_{v-1}+T&\gamma_2\end{pmatrix}+\begin{pmatrix}0_{v\times v}&N\gamma_3\end{pmatrix}=\begin{pmatrix}\vec{e}_1&\begin{pmatrix}0_{1\times(v-1)}&\gamma_1'\\I_{v-1}&\gamma_2'\end{pmatrix}M\end{pmatrix}\]
Notice that the last column only tells how to obtain $\gamma'_1$ and $\gamma'_2$ from $M$, $N$, $\gamma_1$, $\gamma_2$, and $\gamma_3$. We have no restrictions on $\gamma'_1$ and $\gamma'_2$, and they do not appear in equations elsewhere. So we can just ignore this column.

Suppose that $M=\begin{pmatrix}M_1&*\\0&1\end{pmatrix}, M_1\in \mathbb{U}_{v-1}$. Focusing on the left $v$ columns of the above equation, we will get
\[M=\begin{pmatrix}1&0\\0_{1\times(v-1)}&M_1\end{pmatrix}P^{-1} \]
where $P:=\begin{pmatrix}1&a\\0_{(v-1)\times 1}&I_{v-1}+T\end{pmatrix}\in \mathbb{U}_{v}$.
So we know that the $(i+1)$-th column of $M$ is determined by $P$ and the first $i$ columns of $M$. 

Notice that the first column of $M$ must be $\vec{e}_1$. So we know that $M$ is determined by $P$, i.e. by $a$ and $T$. Moreover, if the the last $k+1$ entries of the first $i$-th column of $M$ is zero, we know that the last $k$ entries of the $i+1$-th column of $M$ is zero since $P\in \mathbb{U}_v$. So we know that $M\in \mathbb{U}_v$.

Above all, we have shown there exists a unique $M\in \mathbb{U}_v$ such the left $v\times v$ part of the original equation holds. Then it suffices to show that once we have this $M$, there exists a unique $N$ such that the upper right $v\times v_1$ part of the the original equation holds. 
 Focus on the upper right $v\times v_1$ coordinates. We get the following equation:
   \[M\begin{pmatrix}b\\Q\end{pmatrix}+\begin{pmatrix}N_1B_1\\0\end{pmatrix}=\begin{pmatrix}b'\\N_1\end{pmatrix}\]
The first row of the equation tells us how to determine $b'$ from $M$, $b$, $N_1$, and $B_1$. We have no restriction on $b'$ and it doesn't appear in the equations elsewhere. So we may ignore this row.
   
  Suppose that the $i$-th row of $M\begin{pmatrix}b\\Q\end{pmatrix}$ is $\vec{q_i}$ and the $j$-th row of $N_1$ is $\vec{n_i}$. From the last $v-1$ rows of the above equations, We get that
   \[\vec{n_{v-1}}=\vec{q_v},\vec{n_{i}}=\vec{n_{i+1}}B_1+\vec{q_{i+1}},\quad 1\leq i\leq v-1. \]
    So we know that $N_1$ (which determines $N$) is determined by $b$, $Q$, $B_1$, and $M$. And this is the solution of the upper right $v\times v_1$ part of the original equation.
    
In all, we have shown that for any 
\[ \begin{pmatrix}a &\gamma_1&b\\ I_{v-1}+T&\gamma_2& Q\\ 0_{v_1\times (v-1)}&\gamma_3&B_1\end{pmatrix}\in P_{v_2,v_1}\]
there exists a unique element $u\in U_{v_2,v_1}$  such that \[u\begin{pmatrix}a &\gamma_1&b\\ I_{v-1}+T&\gamma_2& Q\\ 0_{v_1\times (v-1)}&\gamma_3&B_1\end{pmatrix}u^{-1}\in S_{v_2,v_1}\]

\begin{remark}\label{stable}
    Notice that conjugation by $u\in U_{v_2,v_1}$ does not change  $B_1$.
\end{remark}
\end{proof}
Since $GL_{v_2}\times S_{v_2,v_1}$ is a slice of a symplectic reduction, by Lemma~\ref{MW}, when we view this slice as the reduction, the symplectic form is just the pullback from the standard one on $GL_{v_2}\times \mathfrak{gl}_{v_2}$. There is a natural action $\rho$ of $GL_{v_2}\times GL_{v_2}$ on $GL_{v_2}\times \mathfrak{gl}_{v_2}$ given by right multiplication by the inverse and left multiplication:
\[(p,q)(g,a)=(qgp^{-1},ad(q)(a)).\]The corresponding moment map (\cite[Section~2]{Bielawski:1997}) is
\[ (g,a)\to (a,ad(g)a). \]
If we view $GL_{v_2}\times GL_{v_1}$ as a subgroup of $GL_{v_2}\times GL_{v_2}$ in the following way
\[(g_2,g_1)\to (g_2,I_{v_2-v_1}\oplus g_1),\]
the natural action of $GL_{v_2}\times GL_{v_1}$ on the triangle part (viewed as $GL_{v_2}\times S_{v_2,v_1} $) is just the restriction of $\rho$. The moment map is exactly $(B_2,-B_1)$. 

In fact, we can generalize 
Theorem~\ref{base} to a class of handsaw varieties.

\begin{theorem}\label{handsawslice}
    Suppose that $u_n>u_{n-1}>\cdots>u_1 > u_0$. \begin{itemize}
        \item The definition for the handsaw variety is a symplectic reduction with respect to the group $GL_{u_{n-1}}$.
        \item The handsaw quiver variety $HS(u_n,\cdots,u_1,u_0)$ can also be viewed a symplectic reduction of $T^*GL_{u_n}$ by a unipotent group $U_{u_n,\cdots, u_1,u_0}\in GL_{u_n}$.
        \item The above two symplectic structures coming from different symplectic reductions coincide with each other.
        \item The realization we give in Theorem~\ref{construction1} is a slice of both quotients. So, we may view the symplectic form on $HS(u_n,\cdots,u_1,u_0)$ as the pullback of that on $T^*GL_{u_n}$.
        \item  The natural $GL_{v_n}$ action on the handsaw variety is just the restriction of right multiplication action on $T^*GL_{u_n}$, and the moment map attached is just $B_{end}$.
    \end{itemize} 
\end{theorem}
\begin{remark}\label{importantslice}
    Let us view elements in $GL_{u_i}$ (resp. $\mathfrak{gl}_{u_i}$) as elements in $GL_{u_n}$ (resp. $ \mathfrak{gl}_{u_n}$) by the diagonal embedding $I_{u_n-u_i}\oplus GL_{u_i}$ (resp. $0\oplus \mathfrak{gl}_{u_i}$). Then we may take 
    \[f_{u_n,\cdots,u_1,u_0}:=\underset{0\leq i\leq n-1}{\sum}f_{u_{i+1},u_i};\]
    \[U_{u_n,\cdots,u_1,u_0}:=\underset{0\leq i\leq n-1}{\prod}U_{u_{i+1},u_i}.\]
    The latter is well-defined since for $i>j$, $U_{u_{i+1},u_i}$ centralizes $U_{u_{j+1},u_j}$. 
    The symplectic reduction we want to study is just the symplectic reduction by $U_{u_n,\cdots,u_1,u_0}$ at the level $f_{u_n,\cdots,u_1,u_0}$.
  
    There are also maps $res_{u_n,u_i}$ from $\mathfrak{gl}_{u_n}$ to $\mathfrak{gl}_{u_i}$ just restricting to those coordinates in the image of the above embedding. When $n=i$, we take $res_{u_n,u_i}$ to be the identity.

    Then the slice (i.e. the $X$ from Theorem~\ref{construction1}) which we care about is 
    \[S_{u_n,\cdots, u_1,u_0}:=\underset{1\leq i\leq n }{\bigcap}res_{u_n,u_i}^{-1}S_{u_{i},u_{i-1}}.\]
\end{remark}
\begin{proof}[Proof of Theorem~\ref{handsawslice}]

    We prove the above theorem by induction on $n$. 

    Notice that when $n=1$, the statement just follows from the foregoing theorem and the argument after it.

    Suppose we have proved all of these for $HS(u_{n-1},\cdots,u_1,u_0)$. By the last statement, we can see directly $HS(u_{n},\cdots,u_1,u_0)$ is the symplectic reduction of $HS(u_{n},u_{n-1})\times HS(u_{n-1},\cdots,u_1,u_0)$ by the group $GL_{u_{n-1}}$. By the proof of Theorem~\ref{construction1}, we know that the realization we give in Theorem~\ref{construction1} is a slice of this quotient. So it is equivalent to show the following 
    \begin{itemize}
        \item $GL_{u_n}\times S_{u_n,\cdots, u_1,u_0}$ is a slice of the symplectic reduction of $T^*(GL_{u_n})$ at $f_{u_n,\cdots,u_1,u_0}$ by $U_{u_n,\cdots, u_1,u_0}$.
    \end{itemize}
    Take $P_{u_n,\cdots,u_1,u_0}$ as follows:
    \[
    \underset{1\leq i\leq n}{\bigcap } Res_{u_n,u_i}^{-1} P_{u_i,u_{i-1}}
    \]
    Actually, what we want to prove is that for every $K\in P_{u_n,\cdots,u_1,u_0}$, there exists a unique $A\in U_{u_n,\cdots, u_1,u_0}$ such that $AKA^{-1}\in S_{u_n,\cdots, u_1,u_0}$.

    Let us fix $K$ now. 
    
    Notice that $res_{n,n-1}(K)\in P_{u_{n-1},\cdots,u_1,u_0}$. By induction, there exists $M_1\in U_{u_{n-1},\cdots, u_1,u_0}$ such that $M_1KM_1^{-1}\in S_{u_{n-1},\cdots, u_1,u_0}$. View $M_1KM_1^{-1}$ as an element in $\mathfrak{gl}_{u_{n}}$. Notice that it's inside $P_{u_n,u_{n-1}}$. 
    By Theorem~\ref{base}, we know there exists $M_2\in U_{u_n,u_{n-1}}$ such that $M_2(M_1KM_1^{-1})M_2^{-1}\in S_{u_n,u_{n-1}} $.
    From Remark~\ref{stable}, we have  \[res_{u_n,u_{n-1}}(M_2(M_1KM_1^{-1})M_2^{-1})=M_1KM_1^{-1}.\] So we know that \[M_2(M_1KM_1^{-1})M_2^{-1}\ \in S_{u_n,u_{n-1}}\cap res_{u_n,u_{n-1}}^{-1}S_{u_{n-1},\cdots,u_1,u_0}\ =S_{u_n,\cdots,u_1,u_0}\]
    Hence the existence of $A=M_2M_1$ is proved.

    Suppose we have an element $M\in U_{u_n,\cdots ,u_1}$ such that $MKM^{-1}\in S_{u_n,\cdots, u_1,u_0}$. By the definition, we may factor $M$ as $M_1M_2$ in a unique way where $M_1\in U_{u_n,u_{n-1}}$ and $M_2\in U_{u_{n-1},\cdots ,u_1,u_0}$. By Remark~\ref{stable}, we know that 
    \[M_2KM_2^{-1}=res_{u_n,u_{n-1}} (M_1M_2KM_2^{-1}M_1^{-1})\in P_{u_{n-1},\cdots,u_1,u_0} \]
    By induction, $M_2$ is unique. Then, by the uniqueness in the proof of Theorem~\ref{base}, we know $M_1$ is also unique. Hence the uniqueness of $M$ is proven.
\end{proof}
\begin{corollary}
Suppose we are given a bow variety of the form:
\[M(\vec{v},\vec{w})\times _{\pi=B_{end}}HS(u_n,\cdots,u_1,u_0)\sslash GL_{u_n},\ \vec{w}=(w_1,\cdots, w_m), w_m=u_n>\cdots >u_1 > u_0.\]
The variety above, $M(\vec{v},\vec{w})$, is the type $A$ Nakajima quiver variety (stability parameter as $(-1,\cdots,-1)$ with a gauged dimension vector $\vec{v}$ and framed dimension vector $\vec{w}$). The map $\pi$ is just the moment map of $GL_{w_m}=GL_{l_n}$ action on the quiver variety. Then the bow variety can be viewed as
\begin{itemize}
    \item a symplectic reduction of the quiver variety by $U_{u_n,\cdots ,u_1,u_0}$ where $U_{u_n,\cdots ,u_1,u_0}$ acts on framing;
    \item a symplectic subvariety $\pi^{-1}(S_{u_n,\cdots,u_1,u_0})\subset M(\vec{v},\vec{w})$.
\end{itemize}
\end{corollary}
\begin{proof}
    By the foregoing theorem, we are running into the following case:
    \begin{itemize}
        \item $X$ (viewed as $M(\vec{v},\vec{w})$) is a symplectic variety with a Hamiltonian action by $GL_{u_n}$ and a moment map $\pi_1:X\to \mathfrak{gl}_{u_n}$ (view $\mathfrak{gl}_{u_n}$ as $\mathfrak{gl}_{u_n}^*$ via $tr(xy)$).
        \item $GL_{u_n}\times S$ is a symplectic subvariety of $T^*GL_{u_n}=GL_{u_n}\times \mathfrak{gl}_{u_n}$(viewed as $\mathfrak{gl}_{u_n}^*$ by the bilinear form $tr(xy)$) equipped with a $GL_{u_n}$ action as $h(g,s)=(gh^{-1},s)$. This tells us the the moment map $\pi_2$ is just the projection to $S$.
        \item There exists a subgroup $U$ of $\mathbb{U}_{u_n}$ (defined by vanishing of some matrix entries)  and an element $f\in\mathfrak{u}^T$ (idenitified with $\mathfrak{u}^*$ via $tr(xy)$). Then there is a natural projection $p_\mathfrak{u}$ from $\mathfrak{gl}_{u_n}$ to $\mathfrak{u}^T$ by restricting those coordinates and the moment map of $U$ is just $p_\mathfrak{u}\circ \pi_2$. The subvariety $S$ satisfies that for any $y\in p_\mathfrak{u}^{-1}(f)$ there exists a unique $u\in U$ such that $ad(u)y\in S$. 
    \end{itemize}
    
    Due to the first two bullet points, $X\times_{\pi_1(x)=s} S$ is a slice of the symplectic reduction $(X\times_{\pi_1(x)=s} (GL_n\times S))\sslash GL_n$. So we can view the symplectic reduction as $\pi^{-1}(S)$ and the symplectic structure is just the pullback of the symplectic form on $X$.
    
    What we want to prove can be reduced to: 
    $\pi_1^{-1}(S)$ is a slice of the symplectic reduction of $X$ by $U$ at the level $f$.

    Notice that the moment map of $U$ on $X$ is just $p_\mathfrak{u}\circ \pi_1$. Then the level set at $f$ is just $\pi_1^{-1}(p_\mathfrak{u}^{-1}(f) )$. The slice argument is equivalent to saying that for any $(x,\pi_1(x)),x\in X,\pi_1(x)\in p_\mathfrak{u}^{-1}(f)$ there exists a unique $u\in U$ such that $(ux, ad(u)\mu_1(x))\in \pi_1^{-1}(S)$. This comes from the third argument directly.
\end{proof}
Combining the above corollary and the argument in the end of \S\ref{section1}, we reach one of the main results:

\begin{theorem}\label{pointful} (This is a generalized statement of Theorem~\ref{mainthm})
    Every good bow variety is a symplectic reduction of the cotangent bundle of a flag variety by a unipotent group. Moreover, there is a slice of this quotient. So every good bow variety can also be viewed as a symplectic subvariety of the cotangent bundle of flag variety.
\end{theorem}

Here is the explicit description of good bow varieties (including bow varieties with torus fixed points) in the above symplectic reduction language. 

Suppose that we are given two vectors of natural numbers $(\lambda_1,\cdots,\lambda_m)$ and $(\mu_1,\cdots,\mu_n)$. Assume that $\underset{1\leq i\leq m}{\sum}\lambda_m=\underset{1\leq i\leq n}{\sum}\mu_n$ and we will denote the sum by $N$. Then we can construct one nondecreasing sequence
$(l_1,\cdots, l_m)$ and a nonincreasing sequence $(u_n,\cdots, u_1,0)$ where $l_i=\underset{k\leq i}{\sum}\lambda_k$ and $u_i=\underset{k>n-i}{\sum}{\mu_k}$. 

Recall that every good bow variety is determined by such two sequences. By Lemma~\ref{add_1}, we may always ask that $(u_n,\cdots, u_1,0)$ is strictly decreasing, i.e. $\mu_i\geq 1, 1\leq i\leq n$.

Notice that the vector $(l_1,\cdots, l_m)$ determines a flag variety, which we will denote by $T^*(Fl_{\vec{\lambda}})$.

\begin{definition}
     The unipotent group $U_{\vec{\mu}}\subset GL_N$ is defined by the following equations
 \vspace{-0.4em}
\begin{itemize}
\item To describe equations in a better way, we will divide $N\times N$ in to $n\times n$ blocks where $(i,j)$-block is a $\mu_i\times \mu_j$ matrix.
\item When $i>j$, we require that the whole $(i,j)$-block is a zero matrix.
\item When $i=j$, we require that the $(i,i)$-block is an upper-triangular matrix with $1$ on the diagonal.
\item When $i<j$, we require that the last row of the $(i,j)$-block is zero
\end{itemize}
\end{definition}
\begin{example}
    This is exactly the group $U_{(u_n,\cdots, u_1,0)}$
\end{example}
\begin{example}
    When $\vec{\mu}=(2,4)$
    \[U_{\vec{\mu}}=\left\{\begin{pmatrix}1&*&\vline&*&*&*&*\\0&1&\vline&0&0&0&0\\
\hline
0&0&\vline&1&*&*&*\\0&0&\vline&0&1&*&*\\0&0&\vline&0&0&1&*\\0&0&\vline&0&0&0&1\end{pmatrix}:*\in\mathbb{C}\right\}\]
\end{example}
\begin{definition}\label{generalslice}
    By the trace bilinear form $\left<x,y\right>:=tr(xy)$, we may view $\mathfrak{gl}_N$ as $\mathfrak{gl}_N^*$. From this point of view, we may view $\mathfrak{u}_{\vec{\mu}}^*$ as the transpose of $\mathfrak{u}_{\vec{\mu}}$. The map \[\iota^*:\mathfrak{gl}_N\cong \mathfrak{gl}_N^*\to \mathfrak{u}_{\vec{\mu}}^*\cong \mathfrak{u}_{\vec{\mu}}^t\] is just forgetting all other coordinates not appearing in $\mathfrak{u}_{\vec{\mu}}^t$.

Take $e$ to be the matrix where subdiagonals are $1$ and  all other coordinates are $0$. Then $f_\mu$ is defined to be $\iota^*(e)$.
\end{definition}
\begin{remark}
    This is exactly the element $f_{(u_n,\cdots, u_1,0)}$
\end{remark}
\begin{example}
When $\vec{\mu}=(2,4)$,
\[
\mathfrak{u}_{\vec{\mu}}^*=\left\{\begin{pmatrix}&&&&&\\ *&&&&&\\ *&&&&&\\ *&&*&&&\\ *&&*&*&&\\ *&&*&*&*&\end{pmatrix}:*\in \mathbb{C} \right\}
, f_{\vec{\mu}}=\begin{pmatrix}&&&&&\\1&&&&&\\0&&&&&\\0&&1&&&\\0&&0&1&&\\0&&0&0&1&\end{pmatrix}\]

\end{example}
\begin{definition}
    $S_{\vec{\mu}}$ is a subspace of $N\times N$ matrices which can be described blockwisely. We divide the $N\times N$ matrix into $n\times n$ blocks corresponding to the partition $\vec{\mu}$. When $i>j$, we require all the columns except the last one of the $(i,j)$ block to be zero. When $i<j$, we require all the rows except the first one of the $(i,j)$ block  to be zero. We require the $(i,i)$ block to be a companion matrix.
\end{definition}
\begin{example}
    When $\vec{\mu}=(2,4)$, $S_{\vec{\mu}}$ is as follows
    \[\left\{\begin{bmatrix}0&*&\vline &*&*&*&*\\1&*&\vline&0&0&0&0\\
    \hline 0&*&\vline&0&0&0&*\\0&*&\vline&1&0&0&*\\0&*&\vline&0&1&0&*\\0&*&\vline&0&0&1&*\end{bmatrix}:*\in \mathbb{C}\right\}\]
\end{example}

Then the good bow variety labeled by $(l_1,\cdots, l_m)$ and $(u_n,\cdots ,u_1,0)$ (i.e. the bow variety in \cite{RimanyiShou:2020} with row sums $(\lambda_1,\cdots, \lambda_m)$ and columns sums $(\mu_1,\cdots,\mu_n)$) is $T^*(Fl_{\vec{\lambda}})\sslash_{f_{\vec{\mu}}} U_{\vec{\mu}}$. And the slice of this reduction is $\pi_{\vec{\lambda}}^{-1}( S_{\vec{\mu}})$ where $\pi_{\vec{\lambda}}$ is the moment map of $GL_N$ action on $T^*(Fl_{\vec{\lambda}})$. And the affinization of this bow variety is just $\pi_{\vec{\lambda}}(S_{\vec{\mu}})$.

\section{Torus Action on Bow Varieties}\label{torusactiondef}
We also have a nice description of the torus acting on this bow variety in the language of symplectic reduction.

 There is a natural action of $\mathbb{T}\times h\mathbb{C}^*$ on $T^*(Fl_{\vec{\lambda}})$ where $\mathbb{T}$ is the diagonal maximal torus and $h\mathbb{C}^*$ is the dilation action on cotangent fibres. Notice that $\mathbb{T}\times h\mathbb{C}^*$ normalizes $U_{\vec{\mu}}$ (both viewed as subgroups of $Aut(T^*(Fl_{\vec{\lambda}}))$. 
  
 We have a $\mathbb{T}\times h\mathbb{C}^\times$-action on $\mathfrak{u}_{\vec{\mu}}^*$ where $T$ acts by conjugation and $h\mathbb{C}^\times$ acts by scaling (this action does not come from the previous one). This action makes the moment map equivariant.

 Now we have a $Z_{f_{\vec{\mu}}}(\mathbb{T}\times h\mathbb{C}^*)$ action on $T^*(Fl_{\vec{\lambda}})\sslash_{f_{\vec{\mu}}} U_{\vec{\mu}}$ since  $Z_{f_{\vec{\mu}}}(\mathbb{T}\times h\mathbb{C}^*)$ normalizes (resp. centralizes) $U_{\vec{\mu}}$ (resp. $f_{\vec{\mu}}$). We will denote this torus by $\mathbb{T}_{\vec{\lambda},\vec{\mu}}$. This is exactly the torus acting on bow varieties. 
 
 We can decompose $\mathbb{T}_{\vec{\lambda},\vec{\mu}}$ into two parts.

 Denote $\mathbb{T}_{\lambda,\mu,sym}$ to be $Z_{f_{\vec{\mu}}}(\mathbb{T})$, which is 
 \[\left\{z_1 I_{\mu_1}\oplus \cdots \oplus z_n I_{\mu_n}: z_i \in \mathbb{C}^\times \right\}\]
 This is the subtorus of $\mathbb{T}_{\vec{\lambda},\vec{\mu}}$ which acts on bow varieteis Hamiltonianly.

 Consider the cocharacter of $\mathbb{T}\times h\mathbb{C}^*$ given by $w_{\mu_1}\oplus \cdots \oplus w_{\mu_n} \oplus -2$ where $w_{k}$ is the vector $(0,1,\cdots, k-1)$. It's actually a cocharacter of $T_{\mu_i}$. Let us denote the corresponding $\mathbb{C}^\times$  to be $\hat{h}\mathbb{C}^\times$.

 Then we can decompose $\mathbb{T}_{\vec{\lambda},\vec{\mu}}$ as $\mathbb{T}_{\vec{\lambda},\vec{\mu},sym}\times \hat{h}\mathbb{C}^\times$. 

\begin{remark}
    Stable envelopes are important concepts introduced by in \cite{MaulikOkounkov:2019}. We may decompose $U_{\vec{\mu}}$ as $U_1\ltimes U_2$ where $U_1$ is the blockwise diagonal part and $U_2$ is the blockwise strictly uppertriangular part. Suppose the pullback of $f_{\vec{\mu}}$ to $\mathfrak{u}_1^*$ is $f_1$. Then the symplectic reduction $\left(T^*(Fl_{\vec{\lambda}})\right)^{\mathbb{T}_{sym}} \sslash_{f_1} \mathbb{U}_1$ also has a slice which is exactly $Bow^{\mathbb{T}_{\vec{\lambda},\vec{\mu}}}=Bow^{\mathbb{T}_{sym}}$.

    Notice that symplectic reductions give us correpondences which induce the following maps
    \[ \mathcal{L}_U: H_{\mathbb{T}_{\vec{\lambda},\vec{\mu}}}^*(Bow)\to H_{\mathbb{T}_{\vec{\lambda},\vec{\mu}}}^*(T^*Fl_{\vec{\lambda}}),\quad \mathcal{L}_{U_1}:H_{\mathbb{T}_{\vec{\lambda},\vec{\mu}}}^*(Bow^{\mathbb{T}_{\vec{\lambda},\vec{\mu},sym}})\to H_{\mathbb{T}_{\vec{\lambda},\vec{\mu}}}^*\left((T^*Fl_{\vec{\lambda}})^{\mathbb{T}_{\vec{\lambda},\vec{\mu},sym}}\right). \]

    Then we have the following commutative diagram
    \begin{center}
        \begin{tikzcd}
        H_{\mathbb{T}_{\vec{\lambda},\vec{\mu}}}^*(T^*(Fl_{\vec{\lambda}})) &&H_{\mathbb{T}_{\vec{\lambda},\vec{\mu}}}^*\left(T^*(Fl_{\vec{\lambda}})^{\mathbb{T}_{sym}}\right)\arrow[ll,"Stab_{\mathbb{T}_{sym},T^*Fl_{\vec{\lambda}}}"]\\H_{\mathbb{T}_{\vec{\lambda},\vec{\mu}}}^*(Bow)\arrow[u,"\mathcal{L}_{U}"]&&H_{\mathbb{T}_{\vec{\lambda},\vec{\mu}}}^*(Bow^{\mathbb{T}_{sym}})\arrow[ll,"Stab_{\mathbb{T}_{sym},Bow}"] \arrow[u,"\mathcal{L}_{U_{1}}"]
    \end{tikzcd}
    \end{center}
    However, this diagram is not interesting because the map $\mathcal{L}_{U_1}$ is $0$ since  \[T_z\left(T^*(Fl_{\vec{\lambda}})^{\mathbb{T}_{sym}} \right)/ T_z \left( Level(f_1)\right)\] has nontrivial $\mathbb{T}_{\vec{\lambda},\vec{\mu}}$ invariant part for any $z\in Bow^{\mathbb{T}_{sym}}$. On the other hand, if we allow improper pushforward to make a map in the opposite direction for symplectic reductions in the above diagram, we are able to fix the previous problem because both the lie algebra of $U$ and $U_1$ do not have zero weights under the $\mathbb{T}_{\vec{\lambda},\vec{\mu}}$ action. So we can divide the product of the weights. This leads us to consider the following diagram of Lagrangian correspondences
    \begin{center}
        \begin{tikzcd}
        H_{\mathbb{T}_{\vec{\lambda},\vec{\mu}}}^*(T^*(Fl_{\vec{\lambda}})) \arrow[d,"\mathcal{L}_{U}"]&&H_{\mathbb{T}_{\vec{\lambda},\vec{\mu}}}^*\left(T^*(Fl_{\vec{\lambda}})^{\mathbb{T}_{sym}}\right)\arrow[ll,"Stab_{\mathbb{T}_{sym},T^*Fl_{\vec{\lambda}}}"]\arrow[d,"\mathcal{L}_{U_{1}}"]\\H_{\mathbb{T}_{\vec{\lambda},\vec{\mu}}}^*(Bow)&&H_{\mathbb{T}_{\vec{\lambda},\vec{\mu}}}^*(Bow^{\mathbb{T}_{sym}})\arrow[ll,"Stab_{\mathbb{T}_{sym},Bow}"] 
    \end{tikzcd}
    \end{center}
This diagram is commutative and actually gives a geometric explanation of the $D5$ resolution formula in \cite{BottaRimanyi:2023}.
\end{remark}
\section{Equivariant cohomology of \texorpdfstring{$2$}{2}-row bow varieties }\label{section4}
\subsection{Strong deformation retraction}
In this section, we focus on the bow varieties whose torus fixed points are represented by $2$-row matrices. We will call them \defn{$2$-row bow varieties}. According to \cite{RimanyiShou:2020}, these bow varieties are of the following form:
\[T^*Gr(k;u_n)\times_{\pi=B_{end}}HS(u_n,\cdots,u_1,0)\]
where
\begin{itemize}
     \item A sequence $c_i,1\leq i\leq n$ satisfies
\[u_i=\underset{j\leq i}{\sum}c_j.\]
\item $c_i\in\{0,1,2\},\forall 1\leq i\leq n $.
    \item $\max{(k,u_n-k)}\leq \#\{i:c_i\geq 1\}$;
    \item $\min{(k,u_n-k)}\geq\#\{i:c_i=2\}$
    \item $\pi:T^*Gr(k;u_n)\to \mathfrak{gl}_n$ is the moment map for the $GL_{u_n}$ action.
\end{itemize}
Following Theorem~\ref{construction1}, we can give a better description of the above bow variety.

Before that, we need several notations:
\begin{itemize}
\item Introduce the matrix $Nil_{u_n,\cdots, u_1}\in Mat_{u_n\times u_n}$ whose $(u_n-u_i+2,u_n-u_i+1)$-th entry is $1$ if $c_i=2$, and all other entries are $0$.

\begin{example}
$Nil_{3,1,1}=\begin{pmatrix}0&0&0\\1&0&0\\0&0&0\end{pmatrix}$.
\end{example}

\item Introduce the closed subvariety $M_{u_n,\cdots,u_1}$ of $Mat_{u_n\times u_n}$ given by equations:
\[\text{the }(u_n-u_i+2,j)\text{-th entry is zero for }j> u_n-u_i+2\text{ if }c_i=2;  \]
\[\text{the }(j,u_n-u_i+1)\text{-th entry is zero for }j\geq u_n-u_1\text{ if }c_i=2. \]
We denote the embedding $M_{u_n,\cdots,u_1}\hookrightarrow Mat_{u_n\times u_n}$ as $\iota_{u_n,\cdots, u_1}$.

\begin{example}
    $M_{3,1,1}=\left\{\begin{pmatrix}0&a&b\\0&c&0\\0&d&e\end{pmatrix}:\ a,b,c,d,e 
 \in \mathbb{C}\right\}$
\end{example}

\item When $c_i=0$, we have a variety $T_i:=T^*\mathbb{C}^{u_i}$. We have a morphism $\rho_i:T_i\to T^*(\mathbb{C}^{u_n})$
given by:\[\rho_i((\alpha,\beta^T))=\left(\begin{pmatrix}0_{(u_n-u_i)\times 1}\\ \alpha\end{pmatrix},(\begin{pmatrix}0_{(u_n-u_i)\times 1}\\ \beta\end{pmatrix})^T\right),\ \alpha,\beta\in \mathbb{C}^{u_i}\]

\item We have a canonical map $\psi:T^*(\mathbb{C}^{u_n})\to Mat_{u_n\times u_n}$ given by:
\[\psi((\alpha,\beta^T))=\alpha\cdot (\beta^T),\alpha, \beta\in \mathbb{C}^{u_n}\]
\end{itemize}

Now, if we follow the proof in Theorem~\ref{construction1} carefully, we will find that the bow variety above is of the following form:
\[T^*(Gr(k;u_n))\times_{\pi=Nil+\iota+\underset{c_i=0}{\sum}\psi\circ \rho_i }(M_{u_n,\cdots,u_1}\times \underset{c_i=0}{\prod} T_i) \]
\begin{remark}
    Notice that the latter component is isomorphic to $\mathbb{C}^l$ for some $l$.
\end{remark}
Recall the torus action on this bow variety. The torus $\mathbb{T}$ consists of $n$ copies $\mathbb{C}^\times$ each of which is attached to $c_i$ and another global $\mathbb{C}^\times$. We will denote the $\mathbb{C}^\times$ attached to $c_i$ by $\mathbb{C}^\times_i$ and the global $\mathbb{C}^\times$ by $h\mathbb{C}^\times$. The action is as follows:
\begin{itemize}
    \item Notice that $GL_{u_n}$ acts on $T^*(Gr(k,u_n))$ and on $\mathfrak{gl}_{u_n}$ by conjugation. The group $GL_{u_n}$ also acts on $T^*(\mathbb{C}^{u_n})$ by $g(\alpha,\beta^T)=(g\alpha,\beta^Tg^{-1})$.  Once we are given a subtorus of the diagonal matrices in $GL_{u_n}$, we may pull back this action to those $T_i$.
    \item If $c_i=1$, one can embed $\mathbb{C}^\times_i$ into $Z(GL_{u_n})=(\mathbb{C}^*)^{u_n}$ by sending $z$ to $diag(1,\cdots,1,z,1,\cdots ,1)$ where the nontrivial coordinate is the $(u_n-u_i + 1)$-st. The composition of this embedding with the $GL_{u_n}$ action mentioned above is the desired $\mathbb{C}^\times_i$ action.
    \item If $c_i=2$, one can embed $\mathbb{C}^\times_i$ into $Z(GL_{u_n})=(\mathbb{C}^\times)^{u_n}$ by sending $z$ to $diag(1,\cdots,1,z,z,1,\cdots ,1)$ where the nontrivial coordinates are the $(u_n-u_i + 2)$-nd and $(u_n-u_i+1)$-st coordinates. The composition of this embedding with the $GL_{u_n}$ action mentioned above is the desired $\mathbb{C}^\times_i$ action.
    \item  If $c_i=0$, one can embed $\mathbb{C}^\times_i$ into $Z(GL_{u_n})=(\mathbb{C}^\times)^{u_n}$ by sending $z$ to $diag(1,\cdots,1,z,\cdots ,z)$ where the nontrivial coordinates start from the $(u_n-u_i+1)$-st coordinate. The $\mathbb{C}^\times_i$ action on $\mathfrak{gl}_{u_n}$, $T^*(Gr(k,u_n))$ and $T_j,j\neq i$ will be trivial. Its action on $T_i$ will be the composition this embedding with the $GL_{u_n}$ action mentioned above.
    \item There is a $\mathbb{C}^\times$ action on  $T^*(Gr(k,u_n))$ and $T^*(\mathbb{C}^{u_n})$ by dilation of cotangent direction. We can also embed $\mathbb{C}^\times$ into $Z(GL_{u_n})=(\mathbb{C}^\times)^{u_n}$ by sending $z$ to a diagonal matrix whose $(u_n-u_i+1,u_n-u_i+1)$-st coordinate is $z^{-1}$ if $c_i=2$ and all other coordinates are trivial. The above two actions give us a $(\mathbb{C}^\times \times \mathbb{C}^\times)$ action. Viewing $h\mathbb{C}^\times$ as the diagonal will give us the desired action.
\end{itemize}
\begin{remark}
    In fact, we have a simple description of this group action when no $c_i=0$. In this case, the bow variety is just $\pi^{-1}(Nil)$. Notice that there is a torus $(\mathbb{C}^\times)^{u_n}\times \mathbb{C}^\times$ acting on the cotangent bundle. The torus which we take is just $Z_{(\mathbb{C}^\times)^{u_n}\times \mathbb{C}^\times}(Nil)$.
\end{remark}

\begin{example}\label{example1}
When $u_2=u_1=2,\ k=1$ and $n=2$, the corresponding bow variety is 
\[T^*\mathbb{P}^1 \times_{\pi=B} \mathbb{C}^6\] where \[B(x,y,z,w,s,t)=\begin{pmatrix}0&0\\1&0\end{pmatrix}+\begin{pmatrix}0&s\\0&t\end{pmatrix}+\begin{pmatrix}x\\y\end{pmatrix}\begin{pmatrix}z&w\end{pmatrix}=\begin{pmatrix}xz &xw+s\\1+yz&t+yw\end{pmatrix}\]
\end{example}
\begin{example}\label{example2}
    When $u_1=2,u_2=u_3=u_4=4$, $k=2$ and $n=4$, the corresponding bow variety is 
    \[T^*(Gr(2,4) )\times_{\pi=Nil+\iota +\psi\circ \rho_1+\psi\circ \rho_2} (\mathbb{C}^8 \times\mathbb{C}^8 \times \mathbb{C}^8 )\]
    where \begin{itemize}
        \item $Nil=\begin{pmatrix}0&0&0&0\\1&0&0&0\\0&0&0&0\\0&0&1&0\end{pmatrix}$;
        \item 
        $\iota(a_1,\cdots, a_8)=\begin{pmatrix}0&a_1&a_2&a_3\\0&a_4&0&0\\0&a_5&0&a_6\\0&a_7&0&a_8\end{pmatrix}$;
        \item 
        $\psi\circ \rho_1(x_1,\cdots, x_4,y_1,\cdots,y_4)=\begin{pmatrix}x_1\\x_2\\x_3\\x_4\end{pmatrix}\begin{pmatrix}y_1&y_2&y_3&y_4\end{pmatrix} $;
        \item
        $\psi\circ \rho_2(z_1,\cdots, z_4,w_1,\cdots,w_4)=\begin{pmatrix}z_1\\z_2\\z_3\\z_4\end{pmatrix}\begin{pmatrix}w_1&w_2&w_3&w_4\end{pmatrix}$.
    \end{itemize}
\end{example}

Now we want to give a suitable $\mathbb{R}^+$ action on each $2$-row bow variety to apply Theorem~\ref{deformation}.

Notice that, by the previous characterization, such bow variety is of the form
\[T^*(Gr(k,u_n))\times_{\pi=B}(\mathbb{C}^l)  \] where $B$ is a morphism from $\mathbb{C}^l$ to $\mathfrak{gl}_{u_n}$ and $\mathbb{C}^l$ is built up from $M$ (we ignore the index here because there is no ambiguity) and $T_i$. In fact, $B$ is built up from $\iota:M\hookrightarrow \mathfrak{gl}_{u_n}$, $\psi:T^*(\mathbb{C}^{u_n})\to \mathfrak{gl}_{u_n}$ and $\rho_i:T_i\hookrightarrow T^*(\mathbb{C}^{u_n})$. We are going to construct an $\mathbb{R}^+$ action on $T^*(Gr(k,u_n))$, $\mathfrak{gl}_{u_n}$ and $T^*(\mathbb{C}^{u_n} )$ linear on the latter two such that $\mu$ and $\lambda$ are equivariant. Then we will pull back the $\mathbb{R}^+$ action to $M$ and $T_i$ via the map $\iota$ and $\rho_i$ (since these embedding can be viewed as embeddings coming from sub vector space structure). Then, by equivariance, this $\mathbb{R}^+$ action is naturally defined on the bow variety.

To finish the construction, we need the following lemma:

\begin{lemma}\label{Raction}
    Given an integer vector $\vec{v}=(v_1,\cdots,v_{u_n})$, there exists an $\mathbb{R}^+$ action on $T^*(Gr(k,u_n))$, $\mathfrak{gl}_n$ and $T^*(\mathbb{C}^{u_n})$ such that
    \begin{itemize}
        \item The maps $\pi$ and $\psi$ are equivariant with respect to this action.
        \item The $(i,j)$-th entry of $\mathfrak{gl}_n$ is an eigenvector of weight $v_i+2-v_j$.
        \item Consider the standard basis $e_i,1\leq i\leq u_n$ of $\mathbb{C}^{u_n}$ and the dual basis $f_i,i\leq i\leq u_n$. Then $T^*(\mathbb{C}^{u_n})$ can be viewed as a $\mathbb{C}$ vector space with a basis $e_i,f_i,1\leq i\leq u_n$. The $\mathbb{R}^+$ action is just the one making $e_i$ (resp. $f_i$) an eigenvector of weight $v_i$ (resp. $2-v_i$).
    \end{itemize}
\end{lemma}
\begin{proof}
    Notice that the second and the third requirement already determine the $\mathbb{R}^+$ action on $\mathfrak{gl}_{u_n}$ and $T^*(\mathbb{C}^{u_n})$.

    Now let us define the $\mathbb{R}^+$ action on $T^*(Gr(k,u_n))$.

     Let us define $\vec{v}_{mid}:=(v_1-1,\cdots,v_{u_n}-1)$. By viewing $\vec{v}_{mid}$ as a coweight of $(\mathbb{C}^\times)^{u_n}=Z(GL_{u_n})$, we have a map from $\mathbb{R}^+$ to $(\mathbb{C}^\times)^{u_n}=Z(GL_{u_n})\in GL_{u_n} $. Combining this map and the natural action of $GL_{u_n}$ on $T^*(Gr(k,u_n))$, we get an $\mathbb{R}^+$ action on $T^*(Gr(k,u_n))$. Moreover, we have another $\mathbb{R}^+$ action on $T^*(Gr(k,u_n))$ given by rescaling the cotangent vectors. Using these two actions, we get an $\mathbb{R}^+\times \mathbb{R}^+$ action. Then, by embedding $\mathbb{R}^+$ diagonally into $\mathbb{R}^+\times \mathbb{R}^+$, we get the desired $\mathbb{R}^+$ action on $T^*(Gr(k,u_n))$.

     It can be verified directly that the maps $\pi$ and $\psi$ are equivariant with respect to the $\mathbb{R}^+$ action.
\end{proof}

From now on, when we talk about the $\mathbb{R}^+$ action generated by a vector $\vec{v}$, we stick with the action we produced in the above proof. The specific weight vector $\vec{v}=(v_1,\cdots, v_{u_n})$ which will be used to generate the $\mathbb{R}^+$ action in this case is as follows:
\[v_i=\begin{cases}0,&\text{there is a }1\text{ in the }i\text{-th row of }Nil\\2,&\text{there is a }1\text{ in the }i\text{-th column of }Nil\\1,&\text{otherwise} \end{cases}\]

\begin{remark}
The above vector is well-defined since the $Nil$ is blockwise diagonal matrix where the blocks are either $\begin{pmatrix}0\end{pmatrix}$ or $\begin{pmatrix}0&0\\1&0\end{pmatrix}$.
\end{remark}

\begin{example}
    We follow the notation in Example~\ref{example1}. 
    
    The weight vector $\vec{v}$ generating the action is $(2,0)$. The weight of each coordinate is as follows:
    \[wt(x,y,z,w,s,t)=(2,0,0,2,4,2)\]
\end{example}
\begin{example}
    We follow the notation in Example~\ref{example2}. 

    The weight vector $\vec{v}$ generating the action is $(2,0,2,0)$. The weight of each coordinate is as follows:
    \[wt(a_1,a_2,a_3,a_4,a_5,a_6,a_7,a_8)=(4,2,4,2,4,4,2,2)\]
    \[wt(x_1,x_2,x_3,x_4)=wt(z_1,z_2,z_3,z_4)=(2,0,2,0)\]
    \[wt(y_1,y_2,y_3,y_4)=wt(w_1,w_2,w_3,w_4)=(0,2,0,2)\]
\end{example}
Recall that the bow variety in this section is 
\[T^*(Gr(k,u_n))\times_{\pi=B}(\mathbb{C}^l).\]
So we have a natural map:
\[pr:T^*(Gr(k,u_n))\times_{\pi=B}(\mathbb{C}^l)\to \mathbb{C}^l\]
This is just the base change of $\pi$. Since $\pi$ is proper, we know that $\pi$ is proper.

Notice that the standard basis elements of $\mathbb{C}^l$ are eigenvectors of nonnegative weights under the $\mathbb{R}^+$ action.  We can decompose $\mathbb{C}^l$ as $\mathbb{C}^{l_0}\times \mathbb{C}^{l_1}$ where  $\mathbb{C}^{l_0}$ is spanned by vectors of zero weight inside the basis and $\mathbb{C}^{l_1}$ is spanned by vectors of positive weight inside the basis. 

Notice that the data we have collected now satisfy the conditions of Theorem~\ref{deformation}. By this theorem, we know that the bow variety in this section is homotopy equivalent to $pr^{-1}(\mathbb{C}^{l_0})$ via the natural embedding.

\begin{remark}\label{reduction}
    Recall that, under the torus action on the bow variety, the standard basis elements of $\mathbb{C}^l$ are weight vectors. So we know the embedding of $pr^{-1}(\mathbb{C}^{l_0})$ into the bow variety is equivariant with respect to the torus action. Hence, to calculate the equivariant cohomology of the bow variety, it suffices to calculate that of $pr^{-1}(\mathbb{C}^{l_0})$. To give readers an idea on what's going on, when bow variety is a cotangent bundle to a partial flag variety, $pr^{-1}(\mathbb{C}^{l_0})$ is just the zero section.
\end{remark}

Now let us give an explicit description of $pr^{-1}(\mathbb{C}^{l_0})$.
To give a nice description, we need the following notations first:
\begin{itemize}
    \item Take $I$ to be the set of the index $i$ such that $c_i=2$. Give $A$ an enumeration \[I=\{i_1,\cdots,i_p\}\] where $i_1>\cdots>i_p$.
    \item Take $J$ to be the set of the index $i$ such that $c_i=0$ and $i>a_p$.
    Give $J$ an enumeration \[J=\{j_1,\cdots,j_q\}\] where $j_1>\cdots>j_q$.
    \item For $1\leq i\leq q$,  define $r_{i}:=\#\{j:a_j<b_i\}$.
    \item Define $\tilde{T}_i$ to be $T^*(\mathbb{C}^{r_i})=\mathbb{C}^{ r_{i}}\times(\mathbb{C}^{ r_{i}})^* $ for $1\leq i\leq q$.
    \item Define $\tilde{\rho}_i:\tilde{T}_i\to T^*(\mathbb{C}^p)=\mathbb{C}^{p}\times (\mathbb{C}^{p})^*$ by sending $(\alpha,\beta^T)$ to \[\left(\begin{pmatrix}0_{(p-r_i)\times 1}\\\alpha \end{pmatrix},\begin{pmatrix}0_{(p-r_i)\times 1}\\\beta \end{pmatrix}^T\right)\]
    \item Define $\tilde{\psi}: T^*(\mathbb{C}^p) \to Mat_{p\times p} $ by sending $(\alpha,\beta^T)$ to $\alpha\cdot \beta^T$.
\end{itemize}
Then $pr^{-1}(\mathbb{C}^{l_0})$ is isomorphic (via conjugation by a permutation matrix) to 
\[T^*(Gr(k,u_n))\times_{\pi=\begin{pmatrix}0&0\\I_{p}+\underset{1\leq i\leq q}{\sum}\tilde{\psi}\circ \tilde{\rho}_i&0 \end{pmatrix}}\left(\underset{1\leq i\leq q}{\prod}\tilde{T}_i\right)\]

\begin{remark}
Here is the idea on how we derive the above description. Notice that the standard basis of $M_{u_n,\cdots,u_1}$ are of positive weight. So we know that $pr^{-1}(\mathbb{C}^{l_0})$ is a subvariety of 
\[T^*\left(Gr(k;u_n)\right)\times_{\pi=Nil+\underset{c_i=0}{\sum}\psi\circ \rho_i }\left(\underset{c_i=0}{\prod} T_i\right) \]
So we only need to check the weight of each coordinate of $T_i$ to detemrine the $\mathbb{C}^{l_0}$, thus $pr^{-1}(\mathbb{C}^{l_0})$.\end{remark}

We will denote the above subvariety by $Core(k;u_n,\cdots,u_1)$. Similar definitions can be seen in \cite[Section~2]{Kamnitzer:2022} and \cite[Theorem~10.56]{Kirillov:2016} (with a different name there).

One can track the restriction of the torus action on this core part as follows:
\begin{itemize}
    \item Notice that $GL_{u_n}$ acts on $T^*(Gr(k,u_n))$ and $\mathfrak{gl}_{u_n}$ by conjugation. 
    \item There is a torus $\mathbb{T}_{mid}=(\mathbb{C}^\times)^{u_n-2p}$ embedded into $Z(GL_{u_n})$ by sending $(z_1,\cdots,z_{u_n-2p} )$ to $diag(I_p,z_1,\cdots,z_{u_n-2p},I_p)$. The action of $\mathbb{T}_{mid}$ on $T^*(Gr(k,u_n))$ and $\mathfrak{gl}_{u_n}$ is just the composition of the previous embedding and the conjugation. The action of $\mathbb{T}_{mid}$ on $\tilde{T}_i,1\leq i\leq l$ is trivial.  Combining these action, we have a natural action of $\mathbb{T}_{mid}$ on the core.
    \item Notice that there is a $GL_{p}$ action on $T^*(\mathbb{C}^p)=\mathbb{C}^p\times (\mathbb{C}^p)^*$ which is $g(\alpha,\beta^T)=(g\alpha,\beta^T g^{-1})$ . 
    \item There is a torus $\mathbb{T}_{double}=(\mathbb{C}^\times)^p$ embedded into $Z(GL_{u_n})$ by sending $(z_1,\cdots,z_{p} )$ to $diag(z_1,\cdots, z_p,I_{u_n-2p},z_1,\cdots,z_p)$. Then the action of $\mathbb{T}_{double}$ on $T^*(Gr(k,u_n))$ and $\mathfrak{gl}_{u_n}$ just come from the former embedding and the conjugation. The actions of $\mathbb{T}_{double}$ on $T^*(\mathbb{C}^p)$ just comes from the composition the identification of it as $Z(GL_{p})$ with the induced action mentioned before. Then we may pull back this action to those $\tilde{T}^i$ via $\rho_i$. These data gives us the $\mathbb{T}_{double}$ action on the core.
    \item  There is a $\mathbb{C}^\times$ embedded into $Z(GL_{u_n})$ by sending $z$ to $I_{u_n-p}\oplus z^{-1}I_p$. This gives a $\mathbb{C}^\times$ action on $T^*(Gr(k,u_n))$ and $\mathfrak{gl}_{u_n}$. Its action on $T^*(\mathbb{C}^p)$ just comes from $z(\alpha,\beta^T)=(z^{-1}\alpha,\beta^T)$. One can check these data give a $\mathbb{C}^\times$ action on the core. There is another $\mathbb{C}^\times$ action dilating the cotangent fibres of $T^*(Gr(k,n))$ and $T^*(\mathbb{C}^p)$ which induces an action on the core part called dilation action. These two actions give a $\mathbb{C}^\times\times \mathbb{C}^\times$ action. By viewing $h\mathbb{C}^\times$ as the diagonal, we get an $h\mathbb{C}^\times$ action on the core.
    \item For each $i$ such that $c_i=0$, we have that $\mathbb{C}^\times_i$ acts on $\mathfrak{gl}_{u_n}$, $T^*(Gr(k,u_n))$ and $T_j, j\neq i$ trivially. On $T_i=T^*(\mathbb{C}^{r_i})=\mathbb{C}^{r_i}\times (\mathbb{C}^{r_i})^*$, the $\mathbb{C}^\times_i$ acts as $z(\alpha,\beta^T)=(z\alpha,z^{-1}\beta^T)$.
    \item In all, we have a torus $\mathbb{T}=\mathbb{T}_{mid}\times (\underset{c_i=0}{\prod}\mathbb{C}^\times_i)\times \mathbb{T}_{double}\times h\mathbb{C}^\times$ acting on the core. The product of the first three part is just $\mathbb{T}_{sym}$ introduced in \S~\ref{torusactiondef}. 
 Denote $(\underset{c_i=0}{\prod}\mathbb{C}^\times_i)\times \mathbb{T}_{double}\times h\mathbb{C}^\times$ by $\mathbb{T}_{core}$. Later, we will prove that $\mathbb{T}_{core}$-equivariant cohomology of the core part will be concentrated in even degrees. Hence, the $\mathbb{T}$-equivariant cohomology of the core part will be concentrated in even degrees.
\end{itemize}
\begin{remark}
    $\mathbb{T}_{mid}$ is just $\underset{c_i=1}{\prod}\mathbb{C}^\times_i$. $\mathbb{T}_{double}$ is just $\underset{c_i=2}{\prod}\mathbb{C}^\times_i$. 
    There is a natural projection map $\beta:\underset{1\leq i\leq n}{\prod}\mathbb{C}^\times_i\to \underset{c_i=0}{\prod}\mathbb{C}^\times_i$. 
\end{remark}

\begin{example}\label{final1}
    $Core(1;2,2)$ is as follows:
    \[ T^*\mathbb{P}^1\times_{\pi=\begin{pmatrix}0&0\\1+xy&0\end{pmatrix}}(\mathbb{C}\times \mathbb{C})\]
    where $x$ and $y$ are just standard coordiante function on the two copies of $\mathbb{C}$
\end{example}
\begin{example}\label{final2}
    $Core(2;4,4)$ is as follows:
    \[T^*(Gr(2,4))\times_{\pi=B} ((\mathbb{C}^2\times \mathbb{C}^2)\times (\mathbb{C}^2\times \mathbb{C}^2))\]
    where \[B(x_1,x_2,y_1,y_2,z_1,z_2,w_1,w_1)=\begin{pmatrix}0&0&0&0\\0&0&0&0\\1+x_1y_1+z_1w_1&x_1y_2+z_1w_2&0&0\\ x_2y_1+z_2w_1&1+x_2y_2+z_2w_2&0&0\end{pmatrix}.\]
\end{example}
Notice that when $J$ is empty, i.e. there is no $c_i = 0$, we have the following theorem describing the geometry and the equivariant cohomology of the bow variety.
\begin{theorem}\label{firstcal}
    Suppose that $J$ is empty. Then the bow variety is isomorphic to $T^*Gr(k-p,u_n-2p)$. The $\mathbb{T}$-equivariant cohomology of the bow variety is \[H_{ (\underset{c_i=0}{\prod}\mathbb{C}^\times_i)\times \mathbb{T}_{double}\times h\mathbb{C}^\times}^*(pt)\otimes H_{(\mathbb{C}^\times)^{u_n-2p}}^*(Gr(k-p,u_n-2p))\] where $(\mathbb{C}^\times)^{u_n-2p}$ is just the standard torus acting on $Gr(k-p,u_n-2p)$.
\end{theorem}
\begin{proof}
    Notice that when $J$ is empty, the core is isomorphic to:
    \[\pi_{k,u_n}^{-1}(\begin{pmatrix}0&0\\I_p&0\end{pmatrix})\]
    where $\pi_{k,u_n}$ is the moment map on $T^*(Gr(k,u_n))$.

    We can use Springer description of $T^*(Gr(k,u_n))$ as 
    \[\{(V,X):V\subset \mathbb{C}^{u_n}, dim_\mathbb{C} V=k, X\in Mat_{u_n\times u_n}, Im(X)\in V, XV=0 \}.\]
    Use the standard basis to decompose $\mathbb{C}^{u_n}$ as
    \[\mathbb{C}^p\oplus \mathbb{C}^{u_n-2p}\oplus \mathbb{C}^p.\]
    Call the $i$-th component $W_i$
    Then the core is just 
    \[ \{V:V\subset \mathbb{C}^{u_n}, dim_\mathbb{C} V=k, W_3\subset V, V\in W_2\oplus W_3 \}.\]

    So it's isomorphic to \[\{\overline{V}: \overline{V}\subset W_2, dim_\mathbb{C} \overline{W}=k-p \}\cong Gr(k-p,u_n-p)\]

    If we track the $T$ action carefully, we will find that $ (\underset{c_i=0}{\prod}\mathbb{C}^\times_i)\times \mathbb{T}_{double}\times h\mathbb{C}^\times$ acts trivially and $\mathbb{T}_{mid}$ acts just as the maximal torus of $GL_{u_n-p}$. We are done.
\end{proof}
Recall that by Theorem~\ref{pointful}, we know that pointful bow varieties are obtained by symplectic reduction. So one can talk about the Kirwan surjectivity with respect to this quotient. By following Remark~\ref{compare}, Lemma~\ref{add_1}, and the explicit construction of Hanany-Witten isomorphisms in \cite{RimanyiShou:2020}, we have the following corollary.

\begin{corollary}
    Kirwan surjectivity for equivariant cohomology holds for the bow varieties in Theorem~\ref{firstcal}.
\end{corollary}

\subsection{A quick calculation of the ordinary cohomology}\label{calforordcoh}
In this subsection, we calculate the ordinary cohomology of $Core(1;2,2)$. We will find that the odd degree cohomologr does not vanish. So we learn that Kirwan surjectivity does not hold for an ordinary cohomology of bow varieties. 

Recall that $Core(1;2,2)$ is \[ T^*\mathbb{P}^1\times_{\pi=\begin{pmatrix}0&0\\1+xy&0\end{pmatrix}}(\mathbb{C}\times \mathbb{C}).\]

Denote the projection map to $\mathbb{C}\times \mathbb{C}$ by $pr$:

The first open subset $V$ which we would like to take is the set of points where $|xy|<1$. Notice that in this case, $1+xy\neq 0$, $pr$ is a homeomorphism on this open set. So $V$ is just $\{(x,y)\in\mathbb{C}^2:|xy|<1\}$ which is contractible.

The second open subset $U$ we take is the set of points where $x\neq 0$. Denote $1+xy$ by $\tilde{y}$. Notice that $\tilde{y}$ and $x$ determine $y$ on this open set. By reparametrization, this open set is isomorphic to 
\[(T^*\mathbb{P}^1\times_{\pi=\begin{pmatrix}0&0\\\tilde{y}&0\end{pmatrix}}\mathbb{C})\times \mathbb{C}^\times\]
The first part is just the union of $\mathbb{P}^1$ and a fibre over a point in $\mathbb{P}^1$. By the natural scaling on the vector bundle, we know that the first part is homotopy equivalent to $\mathbb{P}^1$.  Thus $U$ is homotopy equivalent to $\mathbb{P}^1 \times \mathbb{C}^\times$

Now consider $U\cap V$. It's just $\{(x,y)\in\mathbb{C}^2:|xy|<1,x\neq 0\}$ which is homotopy equivalent to $\mathbb{C}^\times$.

Considering the long exact sequence attached to this open cover, we have an exact sequence
\[H^3(Core(1;2,2))\to H^3( U)\oplus H^3(V) \to H^3(U\cap V)\]

Notice that $H^3(V)=H^3(U\cap V)=0$ and $H^3(U)\neq 0$. So we know that $H^3(Core(1;2,2))\neq 0$. 

\begin{remark}
    In \cite{Krylov:2021}, it's shown that ordinary cohomology of bow varieties are concentrated in even degrees if the row vector satisfies so-called  ``almost dominant" conditions. What we have just calculated is the very first example breaking this condition and the calculation tells us its ordinary cohomology is not concentrated in even degrees.
\end{remark}
\begin{remark}
Notice that ordinary cohomology of flag varieties always vanishes at odd degrees. Now consider the Kirwan map for bow varieties
\[H^*_{\mathbb{U}}(T^* Fl)\to H^*(Bow) \]
where $\mathbb{U}$ is a unipotent group, in which $H^*_{U}(T^* Fl)\cong H^*(T^* Fl)$. We have just proved that the first one is concentrated in even degrees and the latter one has nonzero cohomology in odd degrees. Hence we know that the Kirwan surjectivity does not in general hold for ordinary cohomology of bow varieties.
\end{remark}
\subsection{An explicit calculation of the equivariant cohomology}\label{calforequco}
In this subsection, we calculate the equivariant cohomology of Example~\ref{final1}. It illustrates that localization to fixed points does not in general induce an injection. This answers a question raised in \cite{RimanyiShou:2020}.

The open cover which we take is \begin{itemize}
    \item the open set $U$ where $x\neq 0$;
    \item the open set $V$ where $\left|xy\right|<1$.
\end{itemize}

Notice that when $x\neq 0$, $1+xy$ determines $y$. So we know that $U$ is isomorphic to \[(T^*\mathbb{P}^1\times_{\pi=\begin{pmatrix}0&0\\\tilde{y}&0\end{pmatrix}}\mathbb{C})\times \mathbb{C}^\times = \mathbb{C}^\times \times \pi^{-1}(\{\begin{pmatrix}0&0\\\tilde{y}&0\end{pmatrix}\}:y\in \mathbb{C})\] where $\pi$ is the moment map for $T^*\mathbb{P}^1$. Recall that $\pi^{-1}(\{\begin{pmatrix}0&0\\\tilde{y}&0\end{pmatrix}\}:y\in \mathbb{C})$ is homotopy equivalent to $\mathbb{P}^1$. So we may view $U$ as $\mathbb{C}^\times \times \mathbb{P}^1$ which is invariant under $\mathbb{T}$.

Notice that when $\left|xy\right|<1$, $\pi$ is an isomorphism. So $V$ is just $\{(x,y)\in \mathbb{C}^2:\left|xy\right|<1\}$ which is contractible. So we may retract it to $(0,0)$, which is fixed by $\mathbb{T}$.

Following the idea working for $U$, we can view $U\cap V$ as $\{(x,y)\in \mathbb{C}^2:\left|xy\right|<1,x\neq 0\}$ which is homotopy equivalent to $\{(x,0):x\neq 0\}=\mathbb{C}^\times$. The latter is also $\mathbb{T}$-invariant.

One can view, in the naive homotopy category, the embedding $U\cap V\to U$ as just sending $\mathbb{C}^\times$ to one of the $\mathbb{C}^\times$ over a $\mathbb{T}$-fixed point on $\mathbb{P}^1$. One can view, in the naive homotopy category, the embedding $U\cap V\to V$ as sending $\mathbb{C}^\times$ to a point.

Now let us review the $\mathbb{T}$ action. In fact, we can decompose $\mathbb{T}$ as $\mathbb{C}^\times_1\times \mathbb{T}_2$ where
\begin{itemize}
    \item $\mathbb{C}^\times_1$ acts nontrivially on $x$.
    \item $\mathbb{T}_2\cong (\mathbb{C}^\times)^2$
    \item $\mathbb{T}_2$ acts trivially on $x$.
    \item $\mathbb{T}_2$ acts on $\mathbb{P}^1$ via the embedding $\mathbb{T}_2\cong (\mathbb{C}^\times)^2\subset GL_2$.
    \end{itemize}

So we know that
\begin{itemize}
    \item $H_\mathbb{T}^*(U)=\mathbb{C}[x_1,x_2,v]/((v-x_1)(v-x_2))$ where $deg(x_1)=deg(x_2)=deg(v) = 2$.
    \item $H_\mathbb{T}^*(V)=\mathbb{C}[x_1,x_2,x_3]$ where $deg(x_i)=2,1\leq i\leq 3$.
    \item $H_\mathbb{T}^*(U\cap V)=\mathbb{C}[x_1,x_2]$ where $deg(x_i)=2, i=1,2$.
    \item The induced ring homomorphism between them is just identifying variables of the same name and sending others to zero.
\end{itemize}

Using the Mayer–Vietoris sequence (\cite[Chapter~3]{AndersonFulton:2023}), we know that the $\mathbb{T}$-equivariant cohomology of this bow variety is 
\[ \mathbb{C}[x_1,x_2,x_3,v]/((v-x_1)(v-x_2),vx_3).\]

Notice that the degree $2$ part is of dimension $4$. On the other hand, there is only one fixed point. So the localization map sends the degree $2$ part to a three dimensional vector space, which can't be injective.

Moreover, equivariant Kirwan surjectivity holds for this bow variety.

\subsection{A general method of calculation}
In this subsection, we calculate the equivariant cohomology of $2$-row bow varieties. 

Recall that, by Remark~\ref{reduction}, it suffices to calculate the $\mathbb{T}$-equivariant cohomology of $Core(k;u_n,\cdots,u_1)$. The core itself is described as follows:

\[T^*(Gr(k,u_n))\times_{\pi=\begin{pmatrix}0&0\\I_{p}+\underset{1\leq i\leq q}{\sum}\tilde{\psi}\circ \tilde{\rho}_i&0 \end{pmatrix}}(\underset{1\leq i\leq q}{\prod}\tilde{T}_i)\]
where
\begin{itemize}
     \item $c_i,1\leq i\leq n$ is a sequence such that 
\[u_i=\underset{j\leq i}{\sum}c_j.\]     
     \item $I$ is the set of the indices $i$ such that $c_i=2$. Give $I$ an enumeration \[I=\{i_1,\cdots,i_p\}\] where $i_1>\cdots>i_p$.
    \item $J$ is the set of the index $i$ such that $c_i=0$ and $i>i_p$.
    Give $J$ an enumeration \[J=\{j_1,\cdots,j_q\}\] where $j_1>\cdots>j_q$.
    \item For $1\leq i\leq q$, $r_{i}:=\#\{j:a_j<b_i\}$.
    \item $\tilde{T}_i$ is $T^*(\mathbb{C}^{r_i})=\mathbb{C}^{ r_{i}}\times(\mathbb{C}^{ r_{i}})^* $ for $1\leq i\leq q$.
    \item  $\tilde{\rho}_i:\tilde{T}_i\to T^*(\mathbb{C}^p)=\mathbb{C}^{p}\times (\mathbb{C}^{p})^*$ 
    sends $(\alpha,\beta^T)$ to \[(\begin{pmatrix}0_{(p-r_i)\times 1}\\\alpha \end{pmatrix},\begin{pmatrix}0_{(p-r_i)\times 1}\\\beta \end{pmatrix}^T)\]
    \item $\tilde{\psi}: T^*(\mathbb{C}^p) \to Mat_{p\times p} $ sends $(\alpha,\beta^T)$ to $\alpha\cdot \beta^T$. 
\end{itemize}
Notice that $\tilde{T}_i=T^*(\mathbb{C}^{r_i})=\mathbb{C}^{r_i}\times (\mathbb{C}^{r_i})^*$. We may view it as 
\[ Spec(\mathbb{C}[x_{i,j},y_{i,j}:1\leq j\leq r_i])\] where $x_{i,j}$ are coordinates on the manifold $\mathbb{C}^{r_i}$ and $y_{i,j}$ are coordinates for the cotangent vectors. To calculate the equivariant cohomology of bow varieties by induction, let us introduce the rank of this core
\[rank(k;u_n,\cdots,u_1):=\underset{1\leq i\leq q}{\sum}r_i\]
Notice that when rank is zero, we run into the case in Theorem~\ref{firstcal}.

Now we are ready to introduce the open covers of the core.

The first open set $U$ is the locus where $x_{q,1}$ is nonzero. This is $\mathbb{T}$-equivariant since $x_{q,1}$ is a weight vector.

The second open set $V$ is the locus where $\left|x_{q,1}y_{q,1}\right|<1$. This is $\mathbb{T}$-equivariant since $x_{q,1}y_{q,1}$ is of weight $0$.

Then we will use this cover to show that

\begin{theorem}\label{secondcal}
    $H^*_{\mathbb{T}_{core}}(Core(k;u_n,\cdots,u_1 ))$ is concentrated in even degrees. 
\end{theorem}

Using the Serre spectral sequence, one can show that

\begin{corollary}
    $H^*_\mathbb{T}(Core(k;u_n,\cdots,u_1 )) $ is concentrated in even degrees.
\end{corollary}

Firstly, we are going to use the following lemma to reduce the above two open sets to bow varieties of smaller size.

\begin{lemma}\label{category}
    Suppose we are working in a category admitting finite limits. Fix $X$, $Y$, $Z$, and a group object $G$. Suppose that we have the following two actions of $G$:
    \[\rho_X:G\times X\to X;\]
    \[\rho_Y:G\times Y\to Y;\]
    a $G$-equivariant morphism:
    \[\phi:X\to Y;\]
    a morphism:
    \[\psi:Z\to Y;\]
    and a morphism
    \[\lambda:Z\to G.\]
    Notice that we have a morphism $\psi_\lambda:Z\to Y$ which is the composition of
    \[\begin{tikzcd}Z\arrow[r,"\lambda\times \psi"]&G\times Y\arrow[r,"\rho_Y"]&Y \end{tikzcd}\]
    
    Then we have a canonical isomorphism:
    \[ \lim\begin{tikzcd}
        X\arrow[r,"\phi"]&Y&Z\arrow[l,"\psi"]
    \end{tikzcd} \cong \lim \begin{tikzcd}X\arrow[r,"\phi"]&Y&Z\arrow[l,"\psi_\lambda"]\end{tikzcd}\]
\end{lemma}
\begin{proof}
Denote the LHS by $L$ and the RHS by $R$.

Notice that a morphism $\alpha:S\to L$ is equivalent to a morphism $\alpha_X\times \alpha_Y\times \alpha_Z:S\to X\times Y\times Z$ such that
\[\phi\circ \alpha_X=\alpha_Y,\quad \psi\circ \alpha_Z=\alpha_Y\]
A similar argument holds for $R$.

Denote the following composition of morphisms by $\gamma_1$:
\[\begin{tikzcd}X\times Y\times Z\arrow[d,"id\times id\times \Delta_Z"]\\X\times Y\times Z\times Z\arrow[d,"id\times id \times id \times \Delta_Z" ]\\ X\times Y\times Z\times Z\times Z\arrow[d,"id\times id\times \lambda\times \lambda "]\\X\times Y\times G\times G\times Z\arrow[d,"\rho_X\times\rho_Y\times id"]\\ X\times Y \times Z\end{tikzcd}\]

One can verify directly that this induces a morphism
\[\gamma'_1:X\times Y\times Z\to R. \]

On the other hand, there is a canonical morphism:$\iota: L\to X\times Y\times Z$. Hence we get a morphism $\tilde{\gamma}_1:L\to R$.

We can use a similar construction to get the inverse of this morphism. Hence we know that $L$ is canonically isomorphic to $R$.
\end{proof}

In practical use of our cases, we will always assume that $X$ is $T^*(Gr(k,u_n)) $, $Y$ is $\mathfrak{gl}_{u_n}$, $\phi$ is the moment map, $Z$ is an open subset of $\underset{1\leq i\leq q}{\prod}\tilde{T}_i$, $\psi$ is the restriction of $\underset{1\leq i\leq q}{\sum}\tilde{\psi}\circ \rho_i$, $G$ is $GL_{u_n}$, and the action is just the one coming from conjugation. 

For the open set $U$, we take $Z$ to be the open subset of $\underset{1\leq i\leq q}{\prod}\tilde{T}_i$ such that $x_{q,1}\neq 0$ and take $\lambda$ to be the map sending $(x_{i,j},y_{i,j})_{1\leq i\leq q,1\leq j\leq r_i}$ to 
\[diag(\begin{pmatrix}1&\beta^T\\0&I_{r_q-1}\end{pmatrix}, I_{u_n-r_q}) \cdot diag (I_{u_n-r_q},\begin{pmatrix}1&0\\ \beta & I_{r_q-1}\end{pmatrix})\]
where $\beta=-\begin{pmatrix}\frac{x_{q,2}}{x_{q,1}}\\\cdots\\\frac{x_{q,r_q}}{x_{q,1}}\end{pmatrix}$. 

After applying Lemma~\ref{category} and picking up another basis of $\tilde{T}_{j},j<q$, we know that the open set is isomorphic to
\[T^*(Gr(k,u_n))\times_{\pi=\begin{pmatrix}0&0\\I_{p}+c_q+\underset{1\leq i<q}{\sum}\tilde{\psi}\circ \tilde{\rho}_i&0 \end{pmatrix}}((\underset{1\leq i< q}{\prod}\tilde{T}_i)\times Spec(\mathbb{C}[x_{q,j},y_{q,j},1/x_{q,1}:1\leq j\leq r_q]))\]
where $c_q=diag(0_{(p-r_q)\times ( p - r_q)},\underset{1\leq j\leq r_q}{\sum}x_{q,i}y_{q,i}, 0_{(r_q-1)\times (r_q - 1)})$. 

Suppose that coordinate functions of $\tilde{T}_i,1\leq i<1$ with respect to the new basis are $\tilde{x}_{i,j},\tilde{y}_{i,j}$. We have the following relations
\[\tilde{x}_{i,l}=x_{i,l}-x_{i,r_i-r_q+1}\frac{x_{q,l-r_i+r_q}}{x_{q,1}},1\leq i<1,l>r_i-r_q+1; \]
\[\tilde{x}_{i,j}=x_{i,j},\text{ otherwise}.\]
\[\tilde{y}_{i,r_i-r_q+1}=y_{i,r_i-r_q+1}+\underset{r_i-r_q<l\leq r_i}{\sum}y_{i,l}\frac{x_{q,l-r_i+r_q}}{x_{q,1}},1\leq i<l;\]
\[\tilde{y}_{i,j}=y_{i,j},\text{ otherwise}.\]

Setting $\vec{v}=(1,2,\cdots,2,1,,\cdots,0)$ where $1$ appears at the first and $u_n-p+1$-th coordinate in Lemma~\ref{Raction} to generate an $\mathbb{R}^+$ action on $T^*(Gr(k,u_n)$ and $\mathfrak{u_n}$. Then we define an $\mathbb{R}^+$ action on $\tilde{x}_{i,j},\tilde{y}_{i,j},1\leq i\leq q,1\leq j\leq r_i$ by the following weight arrangement
\[wt(1+\tilde{x}_{q,1}\tilde{y}_{q,1})=2,wt(\tilde{x}_{q,1})=0;\]
\[wt(\tilde{x}_{q,j})=wt(\tilde{x}_{q,j})=1,2\leq j\leq r_q;\]
\[wt(\tilde{x}_{i,r_i-r_q+1})=wt(\tilde{y}_{i,r_i-r_q+1})=1,1\leq i<q;\]
\[wt(\tilde{x}_{i,j})=wt(\tilde{y}_{i,j})=0,1\leq i<1,2\leq j\leq r_i.\]
One can verify the above $\mathbb{R}^+$ actions will give us an $\mathbb{R}^+$ action on the open set $U$.

Now we may apply Remark~\ref{alternation} to $U$ with respect to the $\mathbb{R}^+$ action we define above. In the naive homotopy category, we can think of this open set as the closed subset defined by the following equations:
\[ 1+\tilde{x}_{q,1}\tilde{y}_{q,1}=0;\]
\[\tilde{x}_{q,j}=\tilde{y}_{q,j}=0,2\leq j\leq r_q;\]
\[\tilde{x}_{i,r_i-r_q+1}=\tilde{y}_{i,r_i-r_q+1}=0,1\leq i<q.\]

Using the relation between $(x_{i,j},y_{i,j})$ and $(\tilde{x}_{i,j},\tilde{y}_{i,j})$, we know that the above closed subset is in fact defined by the the following equations:
\[ 1+x_{q,1}y_{q,1}=0;\]
\[x_{q,j}=y_{q,j}=0,2\leq j\leq r_q;\]
\[x_{i,r_i-r_q+1}=y_{i,r_i-r_q+1}=0,1\leq i<q.\]
Notice that LHS of these equations are $\mathbb{T}$-eigenvectors. So the embedding of this closed subset into $U$ is $\mathbb{T}$-equivariant. Combining this with the fact they are isomorphic via this embedding in the naive homotopy category, to calculate the $\mathbb{T}$-equivariant cohomology of $U$, it suffices to calculate that for the closed subset.

Moreover, using the defining relation involving $x_{i,j},y_{i,j}$, we know that this closed subset is exactly
\[ Core(k;u_{n},\cdots,u_{b_q},u_{b_q-2},u_{b_q-3},\cdots,u_{j},u_{j-1}+1,u_{j-1},\cdots,u_1)\times \mathbb{C}^\times\]
where $j$ is the maximal index such that $c_j=2$ and $j<b_q$. The $\mathbb{C}^\times$ here represents $x_{q,1}$.

Denote the torus $\mathbb{T}_{core}$ attached to $Core(k;u_n,\cdots,u_1)$ as $\mathbb{T}_1$, and the one attached to the above one as $\mathbb{T}_2$. Then we can realize $\mathbb{T}_2$ as a subtorus of $\mathbb{T}_1$ such that 
\begin{itemize}
    \item it acts on $\mathbb{C}^\times$ trivially.
    \item the embedding of that closed subset into $Core(k;u_n,\cdots,u_1)$ is equivariant.
    \item it is of codimension $1$ in $\mathbb{T}_1$.
\end{itemize}
The first and third arguments above tell us that we may identify $\mathbb{C}^\times$ as $\mathbb{T}_2/\mathbb{T}_1$.

By \cite[Example~4.4]{AndersonFulton:2023}, we know that 
\begin{align*}&H^*_{\mathbb{T}_2}Core(k;u_{n},\cdots,u_{b_q},u_{b_q-2},u_{b_q-3},\cdots,u_{j},u_{j-1}+1,u_{j-1},\cdots,u_1)\\=\ &H^*_{\mathbb{T}_1}Core(k;u_n,\cdots,u_1)\end{align*}
Notice that
\begin{align*}&rank(k;u_{n},\cdots,u_{b_q},u_{b_q-2},u_{b_q-3},\cdots,u_{j},u_{j-1}+1,u_{j-1},\cdots,u_1)\\=\ &rank(k;u_n,\cdots,u_1)-r_q-q+1\\<\ &rank(k;u_n,\cdots,u_1)\end{align*}

So we have shown that $\mathbb{T}_{core}$ equivariant cohomology of $U$ is $\mathbb{T}$ equivariant cohomology of a $2$-row bow variety with strictly smaller rank.

\begin{example}
    Consider Example~\ref{final2}. After using Lemma~\ref{category}, we know that the open set $U$ of $Core(1;2,2)$ is isomorphic to 
    \[T^*(Gr(2,4))\times_{\pi=B'} ((\mathbb{C}^2\times \mathbb{C}^2)\times (\mathbb{C}^2\times \mathbb{C}^2))\]
    where \[B'(x_1,x_2,y_1,y_2,z_1,z_2,w_1,w_1)=\begin{pmatrix}0&0&0&0\\0&0&0&0\\\tilde{y}_1+x_2y_2+ z_1\tilde{w}_1&x_1y_2+z_1w_2&0&0\\ \tilde{z}_2\tilde{w}_1&1+\tilde{z}_2w_2&0&0\end{pmatrix};\]
    \[\tilde{y}_1=1+xy\qquad \tilde{z}_2=z_2-\frac{x_2}{x_1}z_1\qquad \tilde{w}_1=w_1+\frac{x_2}{x_1}w_2.\]
The $\mathbb{R}^+$-action we take on this core will act on $T^*(Gr(2,4))$ and $\mathfrak{gl}_4$ by the action induced by vector $(1,2,1,0)$ as in Lemma~\ref{Raction} and act on $\mathbb{C}^8$ with the following weights:
\[wt(\tilde{y}_1,\tilde{w}_1,y_2,z_1,x_2,x_1,\tilde{z}_2,w_2)=(2,1,1,1,1,0,0,0).\]
Then, using Remark~\ref{alternation}, we can view this $U$ as the closed subset of $U$ where $\tilde{y}_1=x_2=y_2=z_1=\tilde{w}_1=0$, which is 
\[\mathbb{C}^\times\times (T^*(Gr(2,4))\times_{\pi=B''}(\mathbb{C}^2))\]
where $\mathbb{C}^\times$ corresponds to $z_1$ and $B''$ is
\[\begin{pmatrix}0&0&0&0\\0&0&0&0\\0&0&0&0\\0&1+z_2w_2&0&0\end{pmatrix}\]
This is $Core(2;4,2,2,1)$.
\end{example}

Now let us consider the open set $V$. We will take $Z\subset \underset{1\leq i\leq q}{\prod}\tilde{T}_i$ to be the open subset where $1+x_1y_1\neq 0$, and take $\lambda$ to be the map sending $(x_{i,j},y_{i,j})_{1\leq i\leq q,1\leq j\leq r_j}$ to 
\[diag(I_{u_n-r_q},\frac{1}{1+x_{q,1}y_{q,1}},I_{r_q-1})\cdot diag(\begin{pmatrix}1&\alpha^T \\0& I_{r_q-1}\end{pmatrix},I_{u_n-r_q})\cdot diag(I_{u_n-r_q},\begin{pmatrix}1&0\\-\beta&I_{r_q-1}\end{pmatrix})\]
where $\alpha=\frac{1}{1+x_1y_1}\begin{pmatrix}x_1y_2\\\cdots\\x_1y_{r_q} \end{pmatrix}$ and $\beta=\frac{1}{1+x_1y_1}\begin{pmatrix}x_2 y_1\\\cdots\\x_{r_q}y_1\end{pmatrix}$.

After applying Lemma~\ref{category} and picking up another basis of $\tilde{T}_j,j<q$, we know that the open set is isomorphic to 
\[\{(x_{q,1},y_{q,1}):\left|x_{q,1}y_{q,1}\right|<1\} \times Core(k:u_n,\cdots,u_{b_q},u_{b_q-2},u_{b_q-3},\cdots,u_j,u_{j-1},u_{j-1},u_{j-2},\cdots,u_1)\]
where $j$ is the maximal index such that $c_j=2$ and $j<b_q$. Notice that the first component is contractible. So, in the naive homotopy category, we may replace $V$ with its closed subset of $V$ defined by $x_{q,1}=y_{q,1}=0$, which is exactly (not via the isomorphism through Lemma~\ref{category}):
\[ Core(k:u_n,\cdots,u_{b_q},u_{b_q-2},u_{b_q-3},\cdots,u_j,u_{j-1},u_{j-1},u_{j-2},\cdots,u_1)\]

Recall that $\mathbb{T}_1$ is the core torus of $Core(k;u_n,\cdots,u_1)$. Notice that both $x_{q,1}$ and $y_{q,1}$ are eigenvectors of this torus, which means the embedding is $\mathbb{T}_1$-equivariant. So the $\mathbb{T}_1$ equivariant cohomology of $V$ is the same as that of 
\[ Core(k:u_n,\cdots,u_{b_q},u_{b_q-2},u_{b_q-3},\cdots,u_j,u_{j-1},u_{j-1},u_{j-2},\cdots,u_1).\]

Denote the $\mathbb{T}_{core}$ of 
\[Core(k:u_n,\cdots,u_{b_q},u_{b_q-2},u_{b_q-3},\cdots,u_j,u_{j-1},u_{j-1},u_{j-2},\cdots,u_1)\]
as $\mathbb{T}_3$. Once we track the torus action, we will find that there is an identification between $\mathbb{T}_1$ and $\mathbb{T}_3$ such that the actions match. So the  $\mathbb{T}_1$-equivariant cohomology of $V$ is exactly the $\mathbb{T}_3$-equivariant cohomology of 
\[ Core(k:u_n,\cdots,u_{b_q},u_{b_q-2},u_{b_q-3},\cdots,u_j,u_{j-1},u_{j-1},u_{j-2},\cdots,u_1).\]
Moreover, the rank of 
\[ Core(k:u_n,\cdots,u_{b_q},u_{b_q-2},u_{b_q-3},\cdots,u_j,u_{j-1},u_{j-1},u_{j-2},\cdots,u_1)\]
is $rank(k;u_n,\cdots,u_1)$. So we reduce the question to a $2$-row bow variety with smaller rank.

Imitating what we have done for $V$, we can show that $U\cap V$ is naively homotopic to the closed subset of $U\cap 
V$ defined by
\[y_{q,1}=0.\]

That's exactly
\[ Core(k:u_n,\cdots,u_{b_q},u_{b_q-2},u_{b_q-3},\cdots,u_j,u_{j-1},u_{j-1},u_{j-2},\cdots,u_1)\times C^\times.\]
Here, $\mathbb{C}^\times$ represents $x_{q,1}$.

Notice that we can find a codimension $1$ subtorus $\mathbb{T}'_3$ of $\mathbb{T}_3$ such that $\mathbb{T}'_3$ acts trivially on $x_{q,1}$ and we may identify $\mathbb{C}^\times$ as $\mathbb{T}_3/\mathbb{T}'_3$. Still, by \cite[Example~4.4]{AndersonFulton:2023}, the $\mathbb{T}_1$-equivariant cohomology of $U\cap V$ is exactly 
\[ H^*_{\mathbb{T}'_3}(Core(k:u_n,\cdots,u_{b_q},u_{b_q-2},u_{b_q-3},\cdots,u_j,u_{j-1},u_{j-1},u_{j-2},\cdots,u_1))\]

If we know that 
\[H^*_{\mathbb{T}_3}(Core(k:u_n,\cdots,u_{b_q},u_{b_q-2},u_{b_q-3},\cdots,u_j,u_{j-1},u_{j-1},u_{j-2},\cdots,u_1))\]
is concentrated in even degrees, by the following lemma, we know the map from 
\[H^*_{\mathbb{T}_3}(Core(k:u_n,\cdots,u_{b_q},u_{b_q-2},u_{b_q-3},\cdots,u_j,u_{j-1},u_{j-1},u_{j-2},\cdots,u_1))\]
to 
\[H^*_{\mathbb{T}_3}(Core(k:u_n,\cdots,u_{b_q},u_{b_q-2},u_{b_q-3},\cdots,u_j,u_{j-1},u_{j-1},u_{j-2},\cdots,u_1))\]
is surjective in even degrees, i.e. $H^*_{\mathbb{T}_1}(V)\to H^*_{\mathbb{T}_1}(U\cap V)$ is surjective in even degrees.

\begin{lemma}\label{ref-appendix-b}
     Suppose we have a space $X$ and a torus $T\times \mathbb{C}^\times$ acting on $X$ such that $H_{T\times \mathbb{C}^\times}^*(X)$ is concentrated at even degrees. Then the map $H_{T\times \mathbb{C}^\times}^{even}(X)\to H_{T}^{evn}(X)$ is surjective.
\end{lemma}
\begin{proof}
See Appendix~\ref{appendix2}.
\end{proof}
Now we are ready to prove Theorem~\ref{secondcal}.
\begin{proof}
    We are going to prove the theorem by induction on the rank.

    When the rank is $0$, we can just use methods in Theorem~\ref{firstcal} to prove the result.

    Suppose the rank is $M>0$ and we have proved the theorem for $2$-row bow varieties of strictly smaller rank.

    Now fix a $2$-row bow variety $Bow$ and use $\mathbb{T}_{core}$ to represent the core torus of it.
    
    We use the open cover constructed before. Notice that by induction we know that both $H^*_{\mathbb{T}_{core}}(U)$ and $H^*_{\mathbb{T}_{core}}(V)$ are concentrated in even degree. It suffices to show the map
    \[H^*_{\mathbb{T}_{core}}(V)\oplus H^*_{\mathbb{T}_{core}}(U)\to H^*_{\mathbb{T}_{core}}(U\cap V)\] is surjective in even degrees. 
    From the argument before the proof, we know that \[H^*_{\mathbb{T}_{core}}(V)\to H^*_{\mathbb{T}_{core}}(U\cap V)\] is already surjective in even degrees. So we have proved the theorem.
    
\end{proof}
\section*{Acknowledgement}
The author would like to express their heartfelt gratitude to Victor Ginzburg and Hiraku Nakajima for their invaluable and detailed responses to the author's inquiries. Additionally, the author extends their sincere thanks to Joel Kamnitzer and Rich\'ard Rim\'anyi for their warm hospitality and enlightening discussions. Special appreciation goes to Matthias Franz for generously offering to contribute an appendix to this paper. Furthermore, the author wishes to acknowledge his advisor, Allen Knutson, for introducing him to bow varieties and for the abundance of helpful discussions, explanations, suggestions, and insightful comments provided throughout the course of this work.

\appendix
\section{Strong deformation retraction}\label{appendix1}
The tricks involved this section are motivated by \cite[Section~4.3]{Slodowy:1980}. We are going to prove a general strong deformation retraction result in the first subsection and then apply it to a special case which will be used frequently in \S~\ref{section4}.
\subsection{General cases}
Let \( X \) be a smooth quasi-projective complex variety, and let \( f : X \to \mathbb{R}^{\geq 0} \) be an analytic function. Suppose there is an action of \( \mathbb{R}^+ \) on \( X \) such that
\[
f(tx) = t^\omega f(x), \quad \text{for all } t \in \mathbb{R}^+ \text{ and } x \in X,
\]
where \( \omega \in \mathbb{Z}^{>0} \).

Notice that $f^{-1}(0)$ is an analytic subset of $X$. By \cite[Section~3~Theorem~2]{Lojasiewicz:1964}, we know that there is a triangulation of the pair $(X,f^{-1}(0))$. Then, by the standard result from algebraic topology (\cite[Page~124]{Spanier:1989}), we know there is an open tubular neighbourhood $U$ of $f^{-1}(0)$ such that the embedding gives a homotopy equivalence.

Suppose that we have a positive continuous function $g$ on $X$ which is invariant under the $\mathbb{R}^+$-action. 

Let's define:
\[h(x):=\max\{t\in [0,1]:g(x)=g(tx)\geq f(tx)=t^\omega f(x)\}.\]

\begin{lemma}
    The function $h$ is continuous.
\end{lemma}
\begin{proof}
    Notice that when $f(x)>0$, we have \[h(x)=\min\{1,(g(x)/f(x))^{1/\omega}\}\] which is the minimum of two continuous functions. So we know that $h$ is continuous on the open set $f^{-1}((0,+\infty))$.

    Given $x\in f^{-1}(0)$, by the continuity of $f$ and $g$, we know there exists a neighborhood $V$ of $x$, such that, for any $y\in V$, we have \[g(y)>\frac{g(x)}{2}>f(y).\] So we know that $h(y)=1$ on this neighborhood $V$. Hence we know $h$ is continuous around $x$. 
    The choice of $x$ is arbitrary. In all, we learn that $h$ is continuous on $X$.
\end{proof}

Now define the following map:
\[\theta(t,x):=\begin{cases}tx,&t>h(x)\\ h(x)x,&t\leq h(x) \end{cases},t\in [0,1],x\in X\]

\begin{lemma}
    $\theta $ is a strong deformation retract from $X$ to $Z_g:=\{x\in X: h(x)=1\}=\{x\in X:g(x)\geq f(x)\}$
\end{lemma}
\begin{proof}
    In fact, it suffices to show $\theta$ is continuous. All other conditions are naturally satisfied.

    Notice that $\theta$ is continuous on $\{(t,x):t>h(x)\}\cup \{(t,x):t<h(x)\}$ since $h$ and the $\mathbb{R}^+$ action is continuous. It suffices to show that for any $x\in X$, $\theta$ is continuous around $(h(x),x)$.

    Suppose that we have a sequence $(t_i,x_i)$ converging to $(h(x),x)$. We may divide it into two subsequences converging to $(h(x),x)$: $(t_{i,1},x_{i,1})$ and $(t_{i,2},x_{i,2})$ where the first sequence is in $\{(t,x):t>h(x)\}$ and the second sequence is in $\{(t,x):t\leq h(x)\}$. It reduces to show that both $\theta(t_{i,1},x_{i,1})$ and $\theta(t_{i,2},x_{i,2})$ converge to $\theta(h(x),x)=h(x)x$.

    Notice that $\theta(t_{i,1},x_{i,1})=t_{i,1}x_{i,1}$. Since the $\mathbb{R}^+$ action is continuous,  it converges to $h(x)x$.

    Notice that $\theta(t_{i,2},x_{i,2})=h(x_{i,2})x_{i,2}$. Since both $h$ and the $\mathbb{R}^+$ action is continuous, we know that this sequence converges to $h(x)x$.

    In all, we have shown $\theta$ is a strong deformation retract.
\end{proof}

Using the above $\theta$, we can show that the embedding of $Z_g$ into $X$ is a homotopy equivalence.

From now on, assume that $Z_g$ is a subset of $U$. 

Notice that the embedding $f^{-1}(0)\to U$ is a homotopy equivalence factoring through $Z_g$. So we know that
$\pi_i(f^{-1}(0))\to \pi_i(U)$ are isomorphisms and $\pi_i(f^{-1}(0))\to \pi_i(Z_g)$ are injections for all $i\geq 0$.

Notice that the embedding $Z_g\to X$ is a homotopy equivalence factoring through $U$. So we know that
$\pi_i(Z_g)\to \pi_i(X)$ are isomorphisms and $\pi_i(U)\to \pi_i(X)$ are surjections for all $i\geq 0$.

Notice that the map $\pi_i(f^{-1}(0))\to \pi_i(X)$ factors as 
$\pi_i(f^{-1}(0))\to\pi_i(Z_g)\to  \pi_i(X)$ (respectively $\pi_i(f^{-1}(0))\to\pi_i(U)\to  \pi_i(X)$) for all $i\geq 0$. So we know that they are surjective (respectively injective) for all $i\geq 0$. 

In all,  we have shown that the embedding $f^{-1}(0)\to X$ induces isomorphisms on $\pi_i,i\geq 0$. As we stated in the beginning of this subsection, we know both of them are CW complexes. So by the Whitehead theorem (\cite{Whitehead:1949}), the embedding is a homotopy equivalence.
\subsection{A special case}
Now let us focus on a case which will be used frequently for us.

Suppose we have a smooth complex quasi-projective variety $X$, an $\mathbb{R}^+$ action on $X$ and an equivariant proper morphism:
\[\pi:X\to Y_1\times Y_2\]
where $Y_1$ is a smooth variety with a trivial $\mathbb{R}^+$ action  and $Y_2=\mathbb{C}^m$ is equipped with a linear $\mathbb{R}^+$ action such that each standard basis is an eigenvector of positive weight.

Suppose that $e_i$ is an eigenvector of positive weight $\omega_i$ for $1\leq i\leq m$.  Let us take $\omega$ to be $\sum_{1\leq i\leq m} \omega_i$ and $x_i$ to be the corresponding monomial in coordinate functions of $e_i$ (if we view $Y_2$ as a vector space, this is just the function in $Y_2^*$ telling the coefficient of the $e_i$) . Consider the following  analytic function on $Y_2$:
\[\tilde{f}(x):=\underset{1\leq i\leq m}{\sum}\norm{x_i}^{\frac{2\omega}{\omega_i}}\]

Notice that we may view $x_i\circ pr_{Y_2}\circ \pi$ as an analytic function on $X$. So $f:=\tilde{f}\circ pr_{Y_2}\circ \pi$ can also be viewed as an analytic function on $X$. Here $pr_{Y_2}$ represents the standard projection from $Y_1\times Y_2$ to $Y_2$.

Notice that we have $f(tx)=t^\omega f(x),x\in X, t\in \mathbb{R}^+$.

Now let us construct the $g$. 

Suppose we are given a compact subset $V\in Y_2$. We are going to find a positive number $b_V$ such that 
\[\{x: f(x)\leq b_V,\pi(x)\in V\times Y_2 \}\subset U.\]

When $\pi^{-1}(V\times Y_2)\subset U$, we can just take $b_V$ to be any positive number. When $\pi^{-1}(V\times Y_2)$ is not a subset of $U$, we can pick up an element $x_V\in U^{c}\cap \pi^{-1}(V\times Y_2)$. Denote $f(x_V)$ by $\lambda$.

Consider the following maps:
\[ \pi|_{V}:\pi^{-1}(V\times Y_2)\to V\times Y_2;\]\[pr_{V}:V\times Y_2\to Y_2;\]
\[ f_V:=\tilde{f}\circ pr_{V}\circ \pi_{V}: \pi^{-1}(V\times Y_2)\to \mathbb{R}^{\geq 0}\]
By the definition, $\tilde{f}$ is proper. Since $V$ is compact, $pr_{V}$ is proper. Notice that, as a restriction of a proper map $\pi$ to a closed subset, $\pi_V$ is also proper. So we know that, as a composition of proper maps,  $f_V$ is a proper map.

Hence $f_V^{-1}([0,\lambda])$ is a compact set. So $U^c\cap f_V^{-1}([0,\lambda])$, denoted as $K_V$, as a closed subset of compact set, is still compact. Since $x_V\in K_V$, we know that $U^c\cap f_V^{-1}([0,\lambda])$ is not empty. Now consider the restriction of $f$ on this nonempty compact set. Then there exists $b_V\geq 0, y_V\in K_V$ such that
\[a_V=\min\{f(x):x\in K_V\}=f(y_V)\]

Since $y_V\in K_V\in U^c$ and $f^{-1}(0)\subset U$, we know that $a_V=f(y_V)>0$.

By the choice of $a_V$, we know that for any $x\in X$ such that $f(x)<a_V$ and $\pi(x)\in V\times Y_2$, we have $x\not\in K_V$, i.e. $x\in U$.
So we can just take $b_V$ to be $\frac{a_V}{2}$.

Since $Y_1$ is a smooth variety, thus a manifold, we can take a locally finite cover $(V_i)_{i\in I}$ of it. In fact, we may further assume the closure $\overline{V_i} $ is compact. Then by the above argument, there exists a number $b_i$ such that
\[ \{x: f(x)\leq b_i,\pi(x)\in \overline{V_i}\times Y_2 \}\subset U.\] Hence we know that 
\[ \{x: f(x)\leq b_i,\pi(x)\in V_i\times Y_2 \}\subset U.\]

Now let us take a partition of unity $(g_i)_{i\in I}$ adapted to this open cover $(V_i)_{i\in I}$ and take $\tilde{g}$ to be
\[\underset{i\in I}{\sum} b_ig_i.\]

By the local finiteness, we know that the above sum makes sense and is continuous. Like $\tilde{f}$, we may pull it back to a positive continuous function $g$ on $X$. Since the $\mathbb{R}^+$ action is trivial, we know that $g$ is $\mathbb{R}^+$-invariant. 

\begin{lemma}
    \[\{x\in X: f(x)\leq g(x) \}\subset U\]
\end{lemma}
\begin{proof}
    Suppose we are given $x\in X$ such that $f(x)\leq g(x) $. Consider the natural projection \[pr_{Y_1}:Y_1\times Y_2.\]
    The point $\tilde{x}:=pr_{Y_1}\circ \pi(x)$ is an element in $X$. By the local finiteness, we know that there exists finitely many indices in $I$:
    \[i_1,\cdots,i_n\] such that $V_{i_j},1\leq j\leq n$ are all the open sets in the open cover containing $\tilde{x}$. 

    Without loss of generality, we may assume that $b_{i_1}\geq \cdots\geq b_{i_n}$. Then we know that
    \[g(x)=\tilde{g}(\tilde{x})=\underset{1\leq j\leq n}{\sum} b_{i_j}g_{i_j}(\tilde{x})\leq b_{i_1}\underset{1\leq j\leq n}{\sum} g_{i_j}(\tilde{x})=b_{i_1}.\]

    Since $\tilde{x}\in V_{i_1}$, we know that $x\in \pi^{-1}(V_{i_1}\times Y_2)$. Since $f(x)\leq g(x)$ and $g(x)\leq b_{i_1}$, we know that $f(x)\leq b_{i_1}$. Then, by the definition of $b_{i_1}$, we have $x\in U$. Hence we have finished the proof.
\end{proof}

Equipped with all the facts we have shown, we know the datum \[(X,f,g,\mathbb{R}^+\text{ action})\] satisfies all the conditions in the previous subsection. So we know that the embedding $f^{-1}(0)\to X$ is a homotopy equivalence. In all, we can summarize it as the following theorem

\begin{theorem}\label{deformation}
    Suppose we have a smooth complex quasi-projective variety $X$, an $\mathbb{R}^+$ action on $X$ and an equivariant proper morphism:
\[\pi:X\to Y_1\times Y_2\]
where $Y_1$ is a smooth variety with a trivial $\mathbb{R}^+$-action  and $Y_2=\mathbb{C}^m$ is equipped with an $\mathbb{R}^+$-action such that each standard basis is an eigenvector of positive weight.

Then the embedding $\pi^{-1}(Y_1\times \{0\})\to X$ is a homotopy equivalence.
\end{theorem}

\begin{remark}\label{alternation}
    In fact, the smoothness of $X$ is not necessary. All we need is that there is a triangulation of the pair $(X,\pi^{-1}(Y_1\times \{0\}))$. Since $X$ is quasi-projective, $X$ can be viewed as an analytic closed subset of a countable analytic manifold $Z$ (some open subset of $\mathbb{A}^k\times \mathbb{P}^l$). Then, still by \cite[Section~3~Theorem~2]{Lojasiewicz:1964}, there exists a triangulation for the triple $(Z,X,\pi^{-1}(Y_1\times \{0\}))$.
\end{remark}
  
\section{Generalized Mirkovic-Vybornov isomorphisims for bow varieties}\label{MVyiso}
Mirkovic-Vybornov isomorphisims were introduced in \cite{MirkovićVybornov:2008}, relating Nakajima quiver varieties, slices in cotangent bundles of flag varieties and slices in affine grassmannians. In \cite{BravermanFinkelbergNakajima:2018}, generalzied affine Grassmannian slices were introduced as Coulomb branches and bow varieties were proved to be Coulomb branches, using factorization properties in \cite{NakajimaTakayama:2017}. In this paper, we have realized bow varieties as slices in cotangent bundles of flag varieties, which is in the favor of classical Mirkovic-Vybornov isomorphisims. We will extend the classical one to the case of bow varieties by giving an explicit isomorphism between slices in cotangent bundles of flag varieties and generalzied affine Grassmannian slices. 
\subsection{Description of generalzied affine grassmannian slices}
Suppose that $\mu=(\mu_1,\cdots,\mu_n)$ where $\mu_i\geq 1$.  

\begin{definition}
The generalized affine grassmannian slice $W_{\vec{\mu}}$ (\cite{KamnitzerPhamWeekes:2022}) is the space with the following two equivalent definitions
\begin{itemize}
    \item 
    
    \[\left\{A=(a_{i,j})_{1\leq i,j\leq n}\in Mat_{n\times n}(\mathbb{C}[x]): \parbox{0.5\columnwidth}{
The southeast $i\times i$  minor is a monic polynomial of degree $N_{n-i+1}$.\\ Take $S=\{1\leq k\leq n:k>i,k>j\}$. The $(\{i\}\cup S)\times (\{j\}\cup S)$ minor is of degree strictly smaller than $N_{\max\{i,j\}}$. }\right\}\]
where $N_i:=\underset{j\geq i}{\sum} \mu_j$.
    \item \[Ux^{\vec{\mu}}T V \cap Mat_{n\times n}(\mathbb{C}[x])\]
    where $U$ is upper-triangular unipotent group whose nontrivial coefficients are in $x^{-1}\mathbb{C}[\![X^{-1}]\!]$, $T$ is the group of diagonal matrices whose coefficients are in $1+x^{-1}\mathbb{C}[\![X^{-1}]\!]$, and  $V$ is the transpose of $U$.
\end{itemize}
\end{definition}

Interested readers may prove it by themselves that $W_{\vec{\mu}}$ is actually an affine space.

Recall that we also have an affine space in a good bow variety determined by two vectors $(\lambda_1,\cdots,\lambda_m)$ and $(\mu_1,\cdots,\mu_n)$ , the slice $S_{\vec{\mu}}$ (\ref{generalslice}). And it's not difficult to show that it's of the same dimension as $W_{\vec{\mu}}$. So there are plenty of isomorphisms between $S_{\vec{\mu}}$ and $W_{\vec{\mu}}$. But we want to construct a specefic isomorphism $\textbf{MVy}$ satisfying the following property:

\begin{itemize}
    \item Denote the lattice \(\mathbb{C}[x]^n\) by \(\mathcal{L}\).
    \begin{itemize}
        \item Given a matrix \(A \in W_{\vec{\mu}}\), denote the lattice generated by the columns of \(A\) as \(\mathcal{L}_A\).
        \item The isomorphism should match:
        \begin{itemize}
            \item The Jordan type of multiplication by \(x\) on \(\mathcal{L}/\mathcal{L}_A\).
            \item That of \(\textbf{MVy}(A)\).
        \end{itemize}
    \end{itemize}
\end{itemize}

To do this, we will find a basis for each $A \in W_{\vec{\mu}}$ and $\textbf{MVy}$ is just representing multiplication by $x$ on \(\mathcal{L}/\mathcal{L}_A\) with respect to this basis. In the next subsection, we will give an explicit construction of this basis and formulas to calculate the MVy isomorphism

Suppose we have this isomorphism of this form. Then it will sends $W_{\vec{\mu}} \cap \pi_{\lambda}(T^*(Fl_{\vec{\lambda}})$ to $S_{\mu}\cap \pi_{\lambda}(T^*(Fl_{\vec{\lambda}})$ where the former is the definition of generalized affine Grassmannian slices and the latter is the affinization of bow varieties. The smooth version is just of the similar form in Springer coordinates. 
\subsection{Explicit calculation of Mirkovic-Vybornov isomorphisims}

\  Recall the two desriptions of $W_{\vec{\mu}}$ introduced the first subsection
\begin{itemize}
    \item 
    
    \[\left\{A=(a_{i,j})_{1\leq i,j\leq n}\in Mat_{n\times n}(\mathbb{C}[x]): \parbox{0.5\columnwidth}{
The southeast $i\times i$  minor is a monic polynomial of degree $N_{n-i+1}$.\\ Take $S=\{1\leq k\leq n:k>i,k>j\}$. The $(\{i\}\cup S)\times (\{j\}\cup S)$ minor is of degree strictly smaller than $N_{\max\{i,j\}}$. }\right\}\]
where $N_i:=\underset{j\geq i}{\sum} \mu_j$.
    \item \[Ux^{\vec{\mu}}T V \cap Mat_{n\times n}(\mathbb{C}[x])\]
    where $U$ is upper-triangular unipotent group whose nontrivial coefficients are in $x^{-1}\mathbb{C}[\![X^{-1}]\!]$, $T$ is the group of diagonal matrices whose coefficients are in $1+x^{-1}\mathbb{C}[\![X^{-1}]\!]$, and  $V$ is the transpose of $U$.
\end{itemize}

Suppose that \( a_{i,j} = -\sum_{l \geq 0} a_{i,j}^{(l)}x^l \). Let us introduce the following notations:
\begin{itemize}
    
    \item For \( b_{n,n}^{(l)} \):
    \begin{itemize}
        \item Set \( b_{n,n}^{(l)} := 0 \) for \( l \geq \mu_n \) and \( b_{n,n}^{(l)} := a_{n,n}^{(l)} \) for \( l < \mu_n \).
    \end{itemize}

    \item For \( n \geq i \geq j \geq 1 \) and \( 0 \leq l \leq \mu_i - 1 \):
    \begin{itemize}
        \item Define \( b_{i,j}^{(l)} := a_{i,j}^{(l)} + \sum_{k > i} \left( \sum_{p+q = \mu_k + l} b_{k,i}^{(p)}b_{i,k}^{(q)} \right) \).
    \end{itemize}

    \item For \( n \geq i > j \geq 1 \) and \( 0 \leq l \leq \mu_i - 1 \):
    \begin{itemize}
        \item Define \( b_{j,i}^{(l)} := a_{j,i}^{(l)} + \sum_{i \leq k \leq n} \left( \sum_{p+q = \mu_k + l} b_{k,i}^{(p)}b_{j,k}^{(q)} \right) \).
    \end{itemize}

    \item For \( 1 \leq i,j \leq n \), \( k = \max(i,j) \), and \( l \geq \mu_k \):
    \begin{itemize}
        \item Define \( c_{i,j}^{(l)} := \delta_{i,j}\delta_{l,\mu_k} + \sum_{s > k} \left( \sum_{p+q = \mu_s + l} b_{s,j}^{(p)}b_{i,s}^{(q)} \right) \),
        where \( \delta_{*,*} \) is the Kronecker symbol.
    \end{itemize}

    \item For \( 1 \leq i,j \leq n \):
    \begin{itemize}
        \item Define \( b_{i,j} = -\sum_{0 \leq l \leq \mu_{\max(i,j)}} b_{i,j}^{(l)} x^l \).
    \end{itemize}

\end{itemize}

Notice that the definition makes sense because it inducts on the following order
\[x^{(k)}_{i,j}> x^{(l)}_{i',j'}\]
if and only if  one of the following two cases hold
\begin{itemize}
    \item $i\leq i'$, $j\leq j' $ and $i+j<i'+j'$ 
    \item $ i=i'$, $j=j'$, and $k<l$.
\end{itemize}

Now we have the following theorem

\begin{theorem}\label{bijectioncoef}
Suppose that $U=(u_{i,j})_{1\leq i,j\leq n}$ and $V=(v_{i,j})_{1\leq i,j\leq n}$. Then we have that 
\[b_{i,j}=\begin{cases}[x^{\mu_j}u_{i,j}],&i<j\\ [\lambda_iv_{i,j}],&i> j\\ [\lambda_{i}-x^{\mu_i}],&i=j\end{cases} \]
and 
\[-b_{i,j}=\begin{cases}[-a_{i,j}+\underset{k\geq j}{\sum}\frac{b_{k,j}b_{i,k}}{x^{\mu_k}}],&i<j\\ [\delta_{i,j}x^{\mu_{i}}-a_{i,j}+\underset{k> j}{\sum}\frac{b_{k,j}b_{i,k}}{x^{\mu_k}}],&i\geq j\end{cases} \]
where $[y]_k$ represent the unique element in $x^k\mathbb{C}[x]\cap (y+x^{k-1}\mathbb{C}[\![x^{-1}]\!])$ for $y\in \mathbb{C}(\!(x^{-1})\!)$. 
\end{theorem}
\begin{remark}
Notice that when $k\geq \max(i,j)$, both $b_{k,i}$ and $b_{j,k}$ are elements in $\mathbb{C}[x]$ of degree strictly smaller than $\mu_k$. So we may replace it with any elements in the same class modulo $x^{-1}\mathbb{C}[\![x^{-1}]\!]$ to calculate $[\frac{b_{k,j}b_{i,k}}{x^{\mu_k}}]_0$.
\end{remark}
\begin{proof}
We are going to prove it by reverse induction on $\max(i,j)$.

Notice that when $i=n$ and $j<n$, $b_{i,j}$, by definition is exactly $a_{i,j}$.

Notice that when $i=j=n$, $b_{n,n}$ by definition is exactly $a_{n,n}-x^{\mu_n}$.

Notice that when $i<n$ and $j=n$, by definition, we have the following identity
\[-b_{i,n}=[-a_{i,n}+\frac{b_{i,n}b_{n,n}}{x^{\mu_n}}]_0\]
Modulo $x^{-1}\mathbb{C}[\![x^{-1}]\!]$, we have \[a_{i,n}=\frac{b_{i,n}(x^{\mu_n}+b_{n,n})}{x^{\mu_n}}.\] Since $x^{\mu_n}+b_{n,n}=a_{n,n}$, modulo $x^{-1}\mathbb{C}[\![x^{-1}]\!]$, we have
\[a_{i,n}=\frac{a_{n,n}b_{i,n}}{x^{\mu_n}}\] which tells us, modulo $x^{\mu_n-1}\mathbb{C}[\![x^{-1}]\!]$, we have 
\[a_{n,n}b_{i,n}=a_{i,n}x^{\mu_n}.\]Since $\deg(a_{n,n})=\mu_n$, modulo $x^{-1}\mathbb{C}[\![x^{-1}]\!]$ we have that\[b_{i,n}=\frac{a_{i,n}}{a_{n,n}}x^{\mu_n},\] i.e.\[ b_{i,n}=[\frac{a_{i,n}}{a_{n,n}}x^{\mu_n}]_0.\]

Direct calculation can show that $v_{n,i}=\frac{a_{n,i}}{a_{n,n}}, 1\leq i<n$, $u_{i,n}=\frac{a_{i,n}}{a_{n,n}},1\leq i<n$, and $a_{n,n}=\lambda_n$. So we have proved all the identities when $\max(i,j)=n$.

Suppose we have shown that all the identities hold for $\max(i,j)\geq s+1$ where $s\leq n-1$. Consider the case that $i=s$.

Suppose that $i=j=s$. Notice that 
\[A=Udiag(\lambda_1,\cdots,\lambda_n)V.\] So we know that 
\[a_{s,s}=\lambda_s+\underset{t>s}{\sum}u_{s,t}\lambda_tv_{t,s}. \]
By induction hypothesis of the identity on $b_{t,s}$ and $b_{s,t}$ for $t>s$, modulo $x^{-1}\mathbb{C}[\![x^{-1}]\!]$, we have 
\[a_{s,s}=\lambda_s+\underset{t>s}{\sum}\frac{b_{s,t}b_{t,s}}{x^{\mu_t}} \]
 The above identity tells us that
\[[a_{s,s}]_{\mu_s}=x^{\mu_s}+\underset{t>s}{\sum}\frac{[b_{s,t}b_{t,s}]_{\mu_t+\mu_s}}{x^{\mu_t}}.\]

By the definition of $b_{s,s}$, we know that, modulo $x^{-1}\mathbb{C}[\![x^{-1}]\!]$,
\[-b_{s,s}=-a_{s,s}+[a_{s,s}]_{\mu_s}+\underset{t>s}{\sum}\frac{b_{s,t}b_{t,s}-[b_{s,t}b_{t,s}]_{\mu_t+\mu_s}}{x^{\mu_t}}.\]

Combining the above two identities, we know that, modulo $x^{-1}\mathbb{C}[\![x^{-1}]\!]$,
\[-b_{s,s}=-a_{s,s}+x^{\mu_s}+\underset{t>s}{\sum}\frac{b_{s,t}b_{t,s}}{x^{\mu_t}}, \] i.e. \[-b_{s,s}=[-a_{s,s}+x^{\mu_s}+\underset{t>s}{\sum}\frac{b_{s,t}b_{t,s}}{x^{\mu_t}}]_0. \]
Taking the identity, modulo $x^{-1}\mathbb{C}[\![x^{-1}]\!]$,  \[a_{s,s}=\lambda_s+\underset{t>s}{\sum}\frac{b_{s,t}b_{t,s}}{x^{\mu_t}} \] into the above identity, we know that \[-b_{s,s}=[-\lambda_s+x^{\mu_s}]_0,\]i.e. \[b_{s,s}=[\lambda_s-x^{\mu_s}]_0\]

Suppose that $i<j=s$. We have
\[a_{i,s}=u_{i,s}\lambda_s+\underset{t>s}{\sum}u_{i,t}\lambda_t v_{t,s}\]
By induction hypothesis, we know that, modulo $x^{-1}\mathbb{C}[\![x^{-1}]\!]$, we have 
\[a_{i,s}=u_{i,s}\lambda_s+\underset{t>s}{\sum}\frac{b_{i,t}b_{t,s}}{x^{\mu_t}} \]
The above identity tells us that
\[[a_{i,s}]_{\mu_s}=\underset{t>s}{\sum}\frac{[b_{i,t}b_{t,s}]_{\mu_t+\mu_s}}{x^{\mu_t}}. \]
By the definition of $b_{i,s}$, we know that, modulo $x^{-1}\mathbb{C}[\![x^{-1}]\!]$,
\[-b_{i,s}=-a_{i,s}+[a_{i,s}]_{\mu_s}+\underset{t\geq s}{\sum}\frac{b_{i,t}b_{t,s}-[b_{i,t}b_{t,s}]_{\mu_t+\mu_s}}{x^{\mu_t}}.\]
Actually, $[b_{i,s}b_{s,s}]_{\mu_s+\mu_s}=0$.
Combining the above three identities, we know that, modulo $x^{-1}\mathbb{C}[\![x^{-1}]\!]$,
\[-b_{i,s}=-a_{i,s}+\underset{t\geq s}{\sum}\frac{b_{i,t}b_{t,s}}{x^{\mu_t}}\]
Taking the identity, modulo $x^{-1}\mathbb{C}[\![x^{-1}]\!]$,\[a_{i,s}=u_{i,s}\lambda_s+\underset{t>s}{\sum}\frac{b_{i,t}b_{t,s}}{x^{\mu_t}} \]
into the above identity, we know that, modulo $x^{-1}\mathbb{C}[\![x^{-1}]\!]$,
\[-b_{i,s}=-u_{i,s}\lambda_s+\frac{b_{i,s}b_{s,s}}{x^{\mu_s}}, \] i.e. modulo $x^{\mu_s-1}\mathbb{C}[\![x^{-1}]\!]$, \[ -b_{i,s}(b_{s,s}
+x^{\mu_s})=-u_{i,s}\lambda_sx^{\mu_s}.\]
Notice that, modulo $x^{-1}\mathbb{C}[\![x^{-1}]\!]$, we have 
\[b_{s,s}
+x^{\mu_s}=\lambda_s.\]
Combining with the fact that $deg(b_{i,s})\leq \mu_s-1$ we know that , modulo $x^{\mu_s-s}\mathbb{C}[\![x^{-1}]\!]$, we have 
\[ -b_{i,s}\lambda_s=-u_{i,s}\lambda_{s} x^{\mu_s}\], i.e.
\[b_{i,s}=[x^{\mu_s}u_{i,s}]_0\]

Similar analysis can be applied to $b_{s,i}$ to finish all the proof.

\end{proof}
\begin{corollary}\label{tech}
When $l\geq\mu_{\max(i,j)}$, we have 
\[-a_{i,j}^{(l)}=c_{i,j}^{(l)}.\]
    
\end{corollary}

\begin{proof}
Just comes from the latter identities in previous theorems.
\end{proof}
Let us call $T_n$ to be the set of the elements $-a_{i,j}^{(l)}-c_{i,j}^{(l)},1\leq i,j\leq n,l\geq \mu_{\max(i,j)}$.

Suppose that the standard basis of $=\mathbb{C}[x]^n$ is $(\vec{e}_i)_{1\leq i\leq n}$. Introduce the following elements of $\mathbb{C}[x]^n$:
\[\vec{v}_{i,0}=\vec{e}_i,\quad 1\leq i\leq n\]
\[\vec{v}_{i,k}=x\vec{v}_{i,k-1}-\underset{j<i}{\sum}b_{j,i}^{(\mu_i-k)}\vec{v}_{j,0},\quad 1\leq k\leq \mu_i-1, 1\leq i \leq n\]

 Call $\vec{b}_{i,j}$ to be the column vector $\begin{pmatrix}b_{i,j}^{(0)}\\b_{i,j}^{(1)}\\\cdots\\b_{i,j}^{\mu_{\max(i,j)}-1}\end{pmatrix}$. If the above set of elements is a $\mathbb{C}$-basis of $\mathcal{L}/\mathcal{L}_A$, multiplication by $x$ with respect to this basis is a matrix of the following form
 \begin{itemize}
    \item The matrix is decomposed blockwise according to the partition \((\mu_1, \cdots, \mu_n)\). Denote the \((i,j)\) block by \(B_{i,j}\).
    \begin{itemize}
        \item For \(i = j\): \(B_{i,j}\) is the companion matrix of \(x^{\mu_i} + b_{i,i}\).
        \item For \(i < j\): \(B_{i,j} = \begin{pmatrix}\vec{b}_{i,j}^T\\0\end{pmatrix}\).
        \item For \(i > j\): \(B_{i,j} = \begin{pmatrix}0&\vec{b}_{i,j}\end{pmatrix}\).
    \end{itemize}
\end{itemize}
This is exactly the description of $S_{\vec{\mu}}$. Moreover, by Theorem~\ref{bijectioncoef}, we know that the construction of $b_{i,j}, 1\leq i,j\leq n$ gives an isomorphism between coefficients $a_{i,j}, 1\leq i,j\leq n$ and those $b_{i,j}, 1\leq i,j \leq n$. This tells us that , once we prove the set of elements listed above is actually a basis, we successfully construct the \textbf{MVy} map, and the map can be explicitly calculated via Theorem~\ref{bijectioncoef}.

Now we are going to prove the elements listed is a basis.

\begin{theorem}\label{basisquotient}
The set $P:=\{\vec{v}_{i,k}:1\leq i\leq n,1\leq k\leq \mu_i\}$ is a basis of the quotient lattice $\mathcal{L}/\mathcal{L}_A$.
\end{theorem}
\begin{proof}
We are going to prove that the vector space (as a vector subspace of $\mathcal{L}/\mathcal{L}_A$) spanned by elements in $P$ is invariant under multiplication of $x$. Notice that $\# P$ is exactly the dimension of $\mathcal{L}/\mathcal{L}_A$. So, we know that  $P$ is a basis.

By the construction of $P$, we know that it suffices to show $x\vec{v}_{i,\mu_i},1\leq i\leq n$ are inside the vector subspace spanned by $P$ inside $\mathcal{L}/\mathcal{L}_A$.

Now consider the following technical lemma:

\begin{lemma}
    Suppose we are dealing with $\mathbb{C}[x,a_{i,j}^{(l_a)},b_{i,j}^{(l_b)},c_{i,j}^{(l_c)}:1\leq i,j\leq n,l_a\geq 0, \mu_{\max(i,j)}>l_b\geq 0,l_c\geq \mu_{\max(i,j)}]^n$ (here we view all these variables as free variables). Assume that the elements in $T_n$ are zero, and the defining relations of $b_{i,j}^{(l)},c_{i,j}^{(l)}$ hold, and defining relations of $\vec{v}_{i,j}$ hold. Then we have that 
\begin{align*}
    x\vec{v}_{i,\mu_i-1} &- \sum_{j=1}^{n} a_{j,i}\vec{e}_j -
     \sum_{l=1}^{i-1} b_{l,i}^{(0)}\vec{v}_{l,0} - \sum_{j=i}^{n} \left( \sum_{k=0}^{\mu_j-1} b_{j,i}^{(k)}\vec{v}_{j,k} \right) = 0, \quad 1 \leq i \leq n.
\end{align*}
\end{lemma}
\begin{proof}
We are going to prove it by induction on $n$.

When $n=1$, $\vec{v}_{1,l}=x^{l}\vec{e}_1,0\leq l\leq \mu_1-1$, $b_{1,1}^{(l)}=a_{1,1}^{(l)},0\leq l\leq \mu_1-1$ and $c_{1,1}^{(l)}=\delta_{l,\mu_1}=-a_{i,j}^{(l)},l\geq \mu_1 $. So we know that 

\begin{align*}
    &\quad\  x\vec{v}_{1,\mu_1-1} - a_{1,1}\vec{e}_1 - \sum_{k=0}^{\mu_1-1} b_{1,1}^{(k)}\vec{v}_{1,k} \\
    &= x^{\mu_1-1}\vec{v}_1 + \sum_{l \geq 0} a_{i,j}^{(l)}x^l\vec{e}_1 - \sum_{k=0}^{\mu_1-1} a_{1,1}^{(k)}\vec{v}_{1,k} \\
    &= 0.
\end{align*}

So the lemma holds when $n=1$.

Consider the case that $n\geq 2$. Suppose the lemma is true for $n-1$.

Consider the case where $i\geq 2$ first.

Let us introduce the following elements.
\[\vec{u}_{i,0}=e_i,\quad 2\leq i\leq n\]
\[\vec{u}_{i,k}=x\vec{u}_{i,k-1}-\sum_{j=2}^{i-1}b_{j,i}^{(\mu_i-k)}\vec{u}_{j,0},\quad 1\leq k\leq \mu_i-1,\ 2\leq i \leq n\]

By induction hypothesis, we know that 
\[x\vec{u}_{i,\mu_i-1}-\sum_{j=2}^{n}a_{j,i}\vec{e}_j-\sum_{l=2}^{i-1}b_{l,i}^{(0)}\vec{v}_{l,0}-\underset{j\geq i}{\sum}\left( \sum_{k=0}^{\mu_j-1} b_{j,i}^{(k)}\vec{v}_{j,k} \right),\quad 2\leq i\leq n. \]
To prove 
\[x\vec{v}_{i,\mu_i-1}-\sum_{j=1}^{n} a_{j,i}\vec{e}_j-\sum_{l=1}^{i-1}b_{l,i}^{(0)}\vec{v}_{l,0}-\underset{j\geq i}{\sum}\left(\sum_{k=0}^{\mu_j-1}b_{j,i}^{(k)}\vec{v}_{j,k}\right)=0,\quad 2\leq i\leq n,\]
it suffices to prove, for $2\leq i\leq n$,
\[x(\vec{v}_{i,\mu_i-1}-u_{i,\mu_i-1})=a_{1,i}\vec{e}_1+\sum_{l=2}^{i-1}b_{l,i}^{(0)}(\vec{v}_{l,0}-\vec{u}_{l,0})+b_{1,i}^{(0)}\vec{v}_{1,0}+\underset{j\geq i}{\sum}\left(\sum_{k=0}^{\mu_j-1}b_{j,i}^{(k)}(\vec{v}_{j,k}-\vec{u}_{j,k})\right)\]
Notice that $v_{l,0}=u_{l,0},\ l\geq 2$. So we know that it suffices to prove 
\[x(\vec{v}_{i,\mu_i-1}-\vec{u}_{i,\mu_i-1})=a_{1,i}\vec{e}_1+b_{1,i}^{(0)}\vec{v}_{1,0}+\underset{j\geq i}{\sum}\left(\sum_{k=0}^{\mu_j-1}b_{j,i}^{(k)}(\vec{v}_{j,k}-\vec{u}_{j,k})\right),\quad 2\leq i\leq n\]
By definition, we have 
\[\vec{v}_{j,k}-\vec{u}_{j,k}=-\sum_{l=0}^{k-1} x^{l} b_{1,j}^{(\mu_j-k+l)} \vec{v}_{1,0},\quad  1\leq k\leq \mu_j-1,\ j\geq 2\]
So it suffices to prove
\begin{align*}
    -x \left( \sum_{l=0}^{\mu_i-2} x^{l} b_{1,i}^{(1+l)} \vec{v}_{1,0} \right) &= a_{1,i}\vec{e}_1 + b_{1,i}^{(0)}\vec{v}_{1,0} \\
    &- \underset{j\geq i}{\sum} \left( \sum_{k=0}^{\mu_j-1} b_{j,i}^{(k)} \left( \sum_{l=0}^{k-1} x^{l} b_{1,j}^{(\mu_j-k+l)} \vec{v}_{1,0} \right) \right), \quad 2 \leq i \leq n.
\end{align*}
Combining with the fact that $\vec{v}_{1,0}=\vec{e}_1$, we only need to prove
\[\sum_{l=1}^{\mu_i-1}x^l b_{1,i}^{(l)}+a_{1,i}+b_{1,i}^{(0)}-\underset{j\geq i}{\sum}\left(\sum_{k=0}^{\mu_j-1}b_{j,i}^{(k)}\left(\sum_{l=0}^{k-1} x^{l} b_{1,j}^{(\mu_j-k+l)}\right)\right)=0,\quad 2\leq i\leq n \]
By rearranging the summation symbol, the above identities are equivalent to the following:
\begin{align*}
    &\sum_{l=0}^{\mu_i-1} x^l \left( b_{1,i}^{(l)} - a_{1,i}^{l} - \sum_{j=i}^{n} \left( \sum_{p+q=\mu_j+l} b_{1,j}^{(p)}b_{j,i}^{(q)} \right) \right) \\
    &+ \sum_{l \geq \mu_i} x^l \left( -a_{1,i}^{l} - \sum_{j=i}^{n} \left( \sum_{p+q=\mu_j+l} b_{1,j}^{(p)}b_{j,i}^{(q)} \right) \right) = 0, \quad 2 \leq i \leq n.
\end{align*}
Since elements in $T_n$ are zero, we know that $-a_{1,i}^{l}=c_{1,i}^{l},\quad l\geq \mu_i$. Since the defining relations of $b_{i,j}^{(l)}$ and $a_{i,j}^{(l)}$ hold, we know that 
\[b_{1,i}^{(l)}=a_{1,i}^{(l)}+\sum_{j=i}^{n}\left(\underset{p+q=\mu_j+l}{\sum}b_{1,j}^{(p)}b_{j,i}^{(q)}\right),\quad 0\leq l\leq \mu_i-1\]
\[c_{1,i}^{l}=\sum_{j=i}^{n}\left(\underset{p+q=\mu_j+l}{\sum}b_{1,j}^{(p)}b_{j,i}^{(q)}\right),\quad l\geq \mu_i.\]
Hence the identities we want to prove are true.

Now consider the case where $i=1$. The identity we want to prove is 
\[x\vec{v}_{1,\mu_1-1}=\sum_{j=1}^{n}a_{j,1}\vec{e}_j+\sum_{j=1}^{n}\left(\sum_{k=0}^{\mu_j-1}b_{j,1}^{(k)}\vec{v}_{j,k}\right).\]
Defining relation of $c_{1,1}^{(\mu_1)}$ tells us that 
\[c_{1,1}^{(\mu_1)}=1+\underset{s>1}{\sum} \left(\underset{p+q=\mu_s+l}{\sum}b_{s,j}^{(p)}b_{i,s}^{(q)}\right)\]
So we know that 
\[x\vec{v}_{1,\mu_1-1}=c_{1,1}^{(\mu_1)}x\vec{v}_{1,\mu_1-1}-\left(\underset{s>1}{\sum} \left(\underset{p+q=\mu_s+l}{\sum}b_{s,j}^{(p)}b_{i,s}^{(q)}\right)\right)x\vec{v}_{1,\mu_1-1}.\]
Since $-c_{1,1}^{(\mu_1)}-a_{1,1}^{(\mu_1)}\in T_n$ and elements in $T_n$ are assumed to be $0$. So we know that 
\[x\vec{v}_{1,\mu_1-1}=-a_{1,1}^{(\mu_1)}x\vec{v}_{1,\mu_1-1}-\left(\underset{s>1}{\sum} \left(\underset{p+q=\mu_s+l}{\sum}b_{s,j}^{(p)}b_{i,s}^{(q)}\right)\right)x\vec{v}_{1,\mu_1-1}.\]
Notice that $v_{1,\mu_1-1}=x^{\mu_1-1}e_1$. Now it suffices to prove 
\begin{align*}
    & -a_{1,1}^{(\mu_1)}x^{\mu_1}\vec{e}_1 - \left( \sum_{s>1} \left( \sum_{p+q=\mu_s+\mu_1} b_{s,j}^{(p)}b_{i,s}^{(q)} \right) \right) x^{\mu_1}\vec{e}_1 \\
    =& \sum_{j=1}^{n} a_{j,1}\vec{e}_j + \sum_{j=1}^{n} \left( \sum_{k=0}^{\mu_j-1} b_{j,1}^{(k)}\vec{v}_{j,k} \right).
\end{align*}
By definition, we have \[\vec{v}_{j,k}=x^k \vec{e}_j-\sum_{l=0}^{k-1}x^l\left( \underset{i<j}{\sum}b_{i,j}^{\mu_i-k+l}\vec{e}_i\right)\]
So we know that
\begin{align*}&\ \sum_{j=1}^{n}\left(\sum_{k=0}^{\mu_j-1}b_{j,1}^{(k)}\vec{v}_{j,k}\right)\\=&\ \sum_{j=1}^{n}\left(\sum_{k=0}^{\mu_j-1}b_{j,1}^{(k)}x^k\right)\vec{e}_j\\ &\ -\sum_{j=1}^{n-1}\left(\underset{k\geq0}{\sum}x^k\left(\sum_{j=i+1}^{n}\left(\underset{p+q=\mu_j+k}{\sum}b_{j,1}^{(p)}b_{i,j}^{(q)}\right)\right)\right)\vec{e}_i\end{align*}

Take this into the identity we want to prove and move right to left. 
Then, the coefficient of each $\vec{e}_i,2\leq i\leq n$ is
\begin{align*}
    &\sum_{k=0}^{\mu_i-1}x^k\left(b_{i,1}^{(k)}-a_{i,1}^{(k)}-\sum_{j=i+1}^{n}\left(\underset{p+q=\mu_j+k}{\sum}b_{j,1}^{(p)}b_{i,j}^{(q)}\right)\right)\\+ &\underset{k\geq\mu_i}{\sum}x^k\left(-a_{i,1}^{(k)}-\sum_{j=i+1}^{n}\left(\underset{p+q=\mu_j+k}{\sum}b_{j,1}^{(p)}b_{i,j}^{(q)}\right)\right)
\end{align*} 
The coefficient of each degree is just the defining relations of $b_{i,1}^{(k)}$ and $-a_{i,1}^{(k)}-c_{i,1}^{(k)}$.
 So it suffices to show the coefficient of $e_1$ is zero. Actually the coefficient can be written as
\begin{align*} &\sum_{k=0}^{\mu_1-1}x^k\left(b_{1,1}^{(k)}-a_{1,1}^{(k)}-\underset{n\geq j>1}{\sum}\left(\underset{p+q=\mu_j+k}{\sum}b_{j,1}^{(p)}b_{1,j}^{(q)}\right)\right)\\+&\underset{k>\mu_1}{\sum}x^k\left(-a_{1,1}^{(k)}-\underset{n\geq j>1}{\sum}\left(\underset{p+q=\mu_j+k}{\sum}b_{j,1}^{(p)}b_{1,j}^{(q)}\right)\right) \end{align*}
which is also a combination of the defining relations of $b_{i,1}^{(k)}$ and $-a_{i,1}^{(k)}-c_{i,1}^{(k)}$. So, as a whole, it’s zero.
\end{proof}
(Continue the proof of Theorem~\ref{basisquotient}.) In $\mathcal{L}/\mathcal{L}_A$, $T_n$ are zero and the defining relations of $b_{i,j}^{(l)},c_{i,j}^{(l)}$ hold. So, in $\mathcal{L}/\mathcal{L}_A$, we have 
\[x\vec{v}_{i,\mu_i-1}=\sum_{l=1}^{i-1}b_{l,i}^{(0)}\vec{v}_{l,0}+\underset{j\geq i}{\sum}\left(\sum_{k=0}^{\mu_j-1}b_{j,i}^{(k)}\vec{v}_{j,k}\right) \] which finishes the proof.
\end{proof}

\input{matthias}

\bibliographystyle{alpha}
\bibliography{ref}
\end{document}

%% file: matthias.tex
\def\C{\mathbb{C}}
\def\HK{H_{K}}
\def\rr{\bar r}
\def\qq{\bar q}

\section{A result in circle-equivariant cohomology (by Matthias Franz)}\label{appendix2}

We consider singular cohomology with coefficients in a commutative ring~$R$ with unit.
The goal of this appendix is to prove the following:

\begin{proposition}
  \label{thm:circle}
  Let $X$ be a space with an action of the group~$K=S^{1}$ or~$K=\C^{\times}$.
  If $\HK^{*}(X)$ is concentrated in even degrees,
  then the restriction map~$\rho\colon\HK^{*}(X)\to H^{*}(X)$ is surjective in even degrees.
\end{proposition}

  We consider the Serre spectral sequence for the bundle~$X\hookrightarrow EK\times_{K}X\to BK$
  with second page
  \begin{equation}
    \label{eq:E2}
    E_{2}^{p,q} = H^{p}(BK;H^{q}(X)) = H^{p}(BK)\otimes_{R} H^{q}(X),
  \end{equation}
  \emph{cf.}~\cite[Sec.~9.4]{Spanier:1989}.
  Since $H^{*}(BK)=R[t]$ is a polynomial ring in one variable of degree~$2$,
  we have $E_{r}^{p,q}=0$ for~$r\ge2$ and odd~$p$, so that the differential~$d_{2r+1}$
  vanishes for~$r\ge1$ and $E_{2r+1}^{p,q}=E_{2r+2}^{p,q}$.
  As this spectral sequence is one of algebras, each row~$E_{2r}^{*,q}$ is a module over~$E_{2}^{*,0}=R[t]$ for~$r\ge1$,
  and each differential~$d_{2r}$ is $R[t]$-linear.

  \begin{lemma}
    Let $\rr\ge1$, and let $\qq$ be odd. Assume that for any~$1\le r<\rr$
    there is no non-zero differential ending on a row~$E_{2r}^{*,q}$ with odd~$q\le \qq$.
    Then for any~$s\ge0$ and~$1\le r\le \rr$, multiplication by~$t^{s}$ is
    \begin{enumroman}
    \item an isomorphism~$E_{2r}^{0,q}\to E_{2r}^{2s,q}$ for odd~$q\le \qq$,
    \item an isomorphism~$E_{2r}^{2r-2,q}\to E_{2r}^{2r-2+2s,q}$ for even~$q\le\qq-2r+3$.
    \end{enumroman}
  \end{lemma}

  \begin{proof}
    We proceed by induction on~$r$.
    For~$r=1$, both claims follow directly from the description~\eqref{eq:E2} of the second page
    of the spectral sequence. Now let $2\le r\le\rr$ and assume both claims proven for~$r-1$.
    
    For~$s\ge0$ and odd~$q\le\qq$ we have a commutative diagram
    \begin{equation*}
      \begin{tikzcd}
	E_{2r-2}^{0,q} \arrow{d}[left]{d_{2r-2}} \arrow{r}{{}\cdot t^{s}} & E_{2r-2}^{2s,q} \arrow{d}{d_{2r-2}} \\
	E_{2r-2}^{2r-2,q-2r+3} \arrow{r}{{}\cdot t^{s}} & E_{2r-2}^{2r-2+2s,q-2r+3}
      \end{tikzcd}
    \end{equation*}
    whose horizontal arrows are isomorphisms by induction. (We can write the bottom arrow as the zigzag
    of isomorphisms
    \begin{equation*}
      \begin{tikzcd}
	E_{2r'}^{2r',q'} & E_{2r'}^{2r'-2,q'} \arrow{l}[above]{{}\cdot t} \arrow{r}{{}\cdot t^{s+1}} & E_{2r'}^{2r'+2s,q'}
      \end{tikzcd}
    \end{equation*}
    with~$r'=r-1$ and~$q'=q-2r+3=q-2r'+1\le\qq-2r'+3$.)
    
    By assumption there is no non-zero differential ending on the odd row~$E_{2r-2}^{*,q}$. 
    Hence we have an isomorphism~$E_{2r}^{0,q}\cong E_{2r}^{2s,q}$ for~$s\ge0$. 
    Also, there is no non-zero differential starting on the even row~$E_{2r-2}^{*,q-2r+3}$
    for otherwise there would be a non-zero differential ending on the row~$q-2r+2\le\qq$ or below it.
    This implies that we have an isomorphism~$E_{2r}^{2r-2,q'}\cong E_{2r}^{2r-2+2s,q'}$ for~$s\ge0$ and~$q'=q-2r+3\le\qq-2r+3$.
    Note that any even row~$q'\le\qq-2r+3$ is covered by this argument since the differential~$d_{2r-2}$
    ending on row~$q'$ starts on the odd row~$q'+2r-3\le\qq$.
  \end{proof}
  
\begin{proof}[Proof of Proposition~\ref{thm:circle}]
  The restriction map~$\rho$ corresponds to the edge homomorphism~$E_{\infty}^{0,*}\to H^{*}(X)$
  of the spectral sequence.
  Assume that $\rho$ is not surjective in even degrees. This means that there is a non-trivial differential
  \begin{equation*}
    \label{eq:dr}
    d_{2r}\colon E_{2r}^{0,q} \to E_{2r}^{2r,q-2r+1}
  \end{equation*}
  for some~$r\ge1$ and some even~$q$.
  We choose $q$ and~$r$ such that $q-2r+1$ is minimal, and among those pairs~$(r,q)$ we pick the one with minimal~$r$.
  Then there is an~$a\in E_{2r}^{0,q}$ with~$d_{2r}(a)\ne0\in E_{2r}^{2r,q-2r+1}$.
  By the first claim of the lemma (with~$\rr=r$ and~$\qq=q-2r+1$),
  we can find a non-zero~$b\in E_{2r}^{0,q-2r+1}$ of odd degree~$q-2r+1$ such that $d_{2r}(a)=t^{r}b$.

  Since $\HK^{*}(X)$ is concentrated in even degrees, there must be an~$r'\ge r$ such that
  $d_{2r'}(b)\ne 0\in E_{2r'}^{2r',q-2r-2r'+2}$, hence also
  \begin{equation*}
    d_{2r'}(t^{r} b) = t^{r}\,d_{2r'}(b)\ne 0 \in E_{2r'}^{2r'+2r,q-2r-2r'+2}
  \end{equation*}
  by the second claim of the lemma with~$\rr\ge r'$ and~$\qq=q-2r-1$ since
  \begin{equation*}
    q -2r -2r' + 2 = \qq - 2r' + 3.
  \end{equation*}

  However, if $r'=r$, this contradicts the fact that $d_{2r}$ is a differential as
  \begin{equation*}
    d_{2r'}(t^{r}b) = d_{2r}d_{2r}(a) = 0.
  \end{equation*}
  If $r'>r$, then we still get a contradiction because $t^{r}b$ vanishes in~$E_{2r+2}^{2r,q-2r+1}$,
  hence also in~$E_{2r'}^{2r,q-2r+1}$. This again forces $d_{2r'}(t^{r}b)=0$ and completes the proof.
\end{proof}

From this we deduce Lemma~\ref{ref-appendix-b} of the main text:

\begin{corollary}
  Let $X$ be a space with an action of~$T\times K$, where $T$ is a torus and $K$ a circle.
  If $H_{T\times K}^{*}(X)$ is concentrated in even degrees, then
  the restriction map~$H_{T\times K}^{*}(X)\to H_{T}^{*}(X)$ is surjective in even degrees.
\end{corollary}

\begin{proof}
  Apply the previous result to the $K$-space~$Y=ET\times_{T}X$ and note that $\HK^{*}(Y)=H_{T\times K}^{*}(X)$.
\end{proof}